\documentclass[11pt, a4paper]{article}
\usepackage{graphicx}
\usepackage{amssymb}
\usepackage{bm}
\usepackage{mathtools}
\usepackage{amsmath}
\usepackage{commath}
\usepackage{amsthm}
\usepackage{enumitem}
\usepackage{float} 
\usepackage{array}
\usepackage{outlines}
\usepackage{setspace}
\usepackage{authblk}
\usepackage[colorinlistoftodos]{todonotes}
\setlist[enumerate,1]{label=\roman*)}
\setlist[enumerate,2]{label=\alph*)}
\newtheorem{theorem}{Theorem}
\newtheorem{corollary}{Corollary}
\newtheorem{lemma}{Lemma}
\newtheorem{remark}{Remark}
\newtheorem{definition}{Definition}
\newtheorem{prop}{Proposition}
\newcommand*\diff{\mathop{}\!\mathrm{d}}
\usepackage{tikz}
\usetikzlibrary{arrows,shapes,trees, calc}

\date{}

\begin{document}

\title{The primitive equations of the polluted atmosphere as a weak and strong limit of the 3D Navier-Stokes equations in downwind-matching coordinates}
\author[1]{\bf{Donatella Donatelli}}
\author[2]{\bf{N\'ora Juh\'asz} \thanks{The research of DD and NJ leading to these results has received funding from the European Union's Horizon 2020 Research and Innovation Programme under the Marie Sklodowska-Curie Grant Agreement No 642768 (Project Name: ModCompShock).}}
\affil[1]{Department of Information Engineering, Computer Science and Mathematics \protect\\ University of L'Aquila, 67100 L'Aquila, Italy \protect\\ \textit{donatella.donatelli@univaq.it} \vspace{0.5cm}}
\affil[2]{Bolyai Institute, University of Szeged, H-6720 Szeged, Hungary \protect\\ \textit{juhaszn@math.u-szeged.hu}}

\maketitle

\begin{abstract}

A widely used approach to mathematically describe the atmosphere is to consider it as a geophysical fluid in a shallow domain and thus to model it using classical fluid dynamical equations combined with the explicit inclusion of an $\epsilon$ parameter representing the small aspect ratio of the physical domain. In our previous paper \cite{pollatmws2019} we proved a weak convergence theorem for the polluted atmosphere described by the Navier-Stokes equations extended by an advection-diffusion equation. We obtained a justification of the generalised hydrostatic limit model including the pollution effect described for the case of classical, east-north-upwards oriented local Cartesian coordinates. Here we give a two-fold improvement of this statement. Firstly, we consider a meteorologically more meaningful coordinate system, incorporate the analytical consequences of this coordinate change into the governing equations, and verify that the weak convergence still holds for this altered system. Secondly, still considering this new, so-called downwind-matching coordinate system, we prove an analogous strong convergence result, which we make complete by providing a closely related existence theorem as well.

\end{abstract}


\newpage

\tableofcontents
\newpage

%%% *** INTRODUCTION *** %%%
\section{Introduction} \label{Introduction}

A local domain on the surface of the rotating Earth is often described in what is called the geophysical rotating frame, i.e. in a Cartesian coordinate system where the $x, y, z$ axes represent the directions towards east, north, and upwards, respectively. This is a widely used coordinate system choice because in this specific system the rotation vector, and as a consequence, the Coriolis force takes a simpler form. However, for modelling wind and pollution effects in the atmosphere there is another naturally distinctive viewpoint for fixing the orientation of the coordinate system: the introduction of the notions of downwind and crosswind, and to match the horizontal axes according to these meteorological parameters (see \cite{goyal2011mathematical}).

Migrating convergence results from one of the above described coordinate systems to the other is not trivial: as a result of an additional $\epsilon$ term in the downwind-matching scenario, certain boundedness and regularity features that are satisfied for the geophysical rotating frame are at the same time not valid for the downwind-matching Cartesian coordinate system.

In \cite{pollatmws2019} we have proved a weak convergence theorem for the case of the classical north-west oriented coordinate system --- the first main part of the present article is dedicated to show that this result still holds true in the downwind-matching coordinate system assuming some medium-level changes in the initial and boundary conditions. The theorem we will migrate between these two coordinate systems is focused on the polluted atmosphere as a shallow domain. Incorporating the shallowness into the model via introducing the 
$$\epsilon = \frac{\text{characteristic depth}}{\text{characteristic width}}$$ 
shallowness parameter, in \cite{pollatmws2019}, we have proved the theoretical correctness of the hydrostatic approximation for the combination of a three-dimensional velocity function and a pollution concentration function. Specifically, we have showed that any weak solution of the shallow polluted atmosphere converges weakly to a weak solution of the so-called primitive equations representing the hydrostatic limit model. In the first part of this article our goal is to verify an analogous version of this convergence theorem for the case of the more considerately chosen downwind-matching coordinate system. As in this scenario the advection is dominant compared to diffusion effects along the $x$ axis, this approach on the one hand naturally contributes to the simplification of the limit model. On the other hand it brings an additional $\epsilon$ term in the Laplacian of the concentration equation, and along with this the challenge of proving a weak convergence result for the product of the vertical velocity and the concentration. The difficulty is due to the phenomenon of the $\norm{\partial_1 c^\epsilon}$ term being replaced by $\| \sqrt{\epsilon} \partial_1 c^\epsilon \|$ in the energy inequality, where $c^\epsilon$ is the pollution concentration. As the a priori estimates extracted from the energy inequality serve as fundamental building blocks of the proof, having a weakened, $\epsilon$-including version of one of the estimates has significant consequences for the verification process --- we lose the strong convergence of the concentration function. This issue is overcome by using a more regular initial horizontal velocity function, namely ${\mathbf{u}_h}_0 \in H^1.$ The stronger initial regularity will naturally result in a $\mathbf{u}^\epsilon$ function with stronger regularities, allowing us to pass to the field of strong solutions for the velocity part of the solution, and thus will compensate the lost $H^1$ regularity of the concentration function $c^\epsilon.$ 
%We mention that we will also use slightly different, stricter boundary conditions for the concentration function --- these additional requirements too are motivated by the presence of the new $\epsilon$ term in the energy inequality.

In the second part of this article we investigate the connection between the anisotropic and hydrostatic models of the polluted atmosphere on the level of strong solutions. The downwind-matching coordinate system and the set of modelling equations we work with are essentially identical to those defined in part one --- what separates the two main sections is the quality of the convergence we are able to state. We begin the second part of this paper by examining in detail what modifications and changes are necessary in the model so that we can prove a convergence theorem analogous to that formulated in the first part, but stated for the case of strong solutions this time. The higher order of derivatives due to considering strong solutions would mean more complicated integral terms on the domain's boundary, hence in this case we describe a method (see also \cite{cao2014local}) that justifies the definition of a domain that is periodic not just along the horizontal variables, but in the vertical direction as well (note that this, considering the natural structure of the atmosphere, is not trivial). Concerning the regularity of the initial data, our proofs require $c_0$ to be in $H^2$ and ${\mathbf{u}_h}_0$ to be in $H^3$, the relatively high regularity of the initial horizontal velocity is necessary due to the presence of the additional $\epsilon$ multiplier in the $K_x$ diffusion term. In other words, similarly to the phenomenon described for the weak solutions, in this case we also need to tailor our proofs carefully at the points where the downwind-matching coordinate choice affects the estimates: it is these modifications that require higher level of regularities. 
In more detail, following otherwise standard procedures of multiplying the equations by Laplacians and integrating in space, in this specific coordinate system we obtain only $\epsilon \norm{\partial_{11} c^\epsilon}$ on the equation's left-hand side. 
%This means that the usual strategy of bringing high-order terms multiplied by a Young inequality coefficient to the left-hand side will not work in the case of the downwind-matching coordinate system. 
We overcome this issue by constructing different types of estimates where the $\epsilon$-free $\partial_{11} c^\epsilon$ term does not appear in the final bound, and by exploiting the increased $H^3$ regularity of ${\mathbf{u}_h}_0$ which guarantees an increased regularity for the velocity solutions as well.

This paper is organised in the following way. In Section \ref{eqsdescribingatmpluspollution} we describe the downwind-matching coordinate system, the physical model of the polluted atmosphere and the scaling leading to the hydrostatic equations. In Section \ref{sectionWS} we state and prove the main convergence result regarding weak solutions, while in Section \ref{sectionSTRONGsol} we elaborate the analogous result for strong solutions. Finally, in order to make the strong convergence result of the latter section complete, we prove an existence theorem on strong solutions regarding the hydrostatic system in the Appendix.

We end up this section with some notations we will use through the paper.

\subsection{Notations}

\begin{enumerate}

\item $\epsilon = \frac{\text{characteristic depth}}{\text{characteristic width}} = $ the aspect ratio,
\item $\nu = $ viscosity,
\item $\Omega^\epsilon, \Omega = $ the original and the rescaled ($\epsilon$-independent) domain, respectively,
\item the constant $a$ denotes the height of the domain (with the flexibility that depending on the case this could either mean the three-dimensional domain is composed as $\Omega_2 \times [0,a]$ or $\Omega_2 \times [-a, a]$),
\item $\nabla$ stands for the gradient vector; specifically, $(\partial_x, \partial_y, \partial_z)$ before, and $(\partial_1, \partial_2, \partial_3)$ after rescaling,
\item $\nabla_\nu = (\nu_x^{1/2} \partial_x, \nu_y^{1/2} \partial_y, \nu_z^{1/2} \partial_z)$ on $\Omega^\epsilon$, $\nabla_\nu = (\nu_1^{1/2} \partial_1, \nu_2^{1/2} \partial_2, \nu_3^{1/2} \partial_3)$ on the rescaled domain $\Omega,$
\item $ \Delta_\nu$ stands for the anisotropic Laplacian operators $\nu_x \partial^2_{xx} + \nu_y \partial^2_{yy} + \nu_z \partial^2_{zz}$ and $\nu_1 \partial^2_{11} + \nu_2 \partial^2_{22} + \nu_3 \partial^2_{33}$ on the original and rescaled domains, respectively,
\item $\mathbf{g}$ represents the force due to gravity, which also includes the centrifugal effect,
\item $\varphi = $ the gravity potential term, i.e. $\mathbf{g} = \nabla \varphi,$
\item $\mathbf{\omega}$ represents the Earth rotation angular speed,
\item $\phi = l(y) = $ latitude, $f = $ the module of the Earth rotation vector
\item $\alpha = -2 f \sin(\theta)\cos(\phi), \beta = 2 f \cos(\theta) \cos(\phi), \gamma = 2 f \sin(\phi),$
\item $\mathbf{K}$ represents the diagonal diffusion matrix,
\item $( \cdot,\cdot) = $ the scalar product in $L^2(\Omega)^d, d \geq 1$ or the duality $L^p (\Omega)$, $L^{p'} (\Omega),$
\item $\mathbf{u}_h$ = the horizontal velocity components $(u_1, u_2),$
%\item $b(\mathbf{u}_h)$ = $\gamma ( - u_2 , u_1),$
\item $\Delta_d = K_2 \partial_{22}^2 + K_3 \partial_{33}^2$
\item $\nabla_d = (\partial_2, \partial_3)$ ($d$ stands for directions in which we have diffusion),
\item $L^p_t L^q_x$ is the abbreviated notation for $L^p (0, T; L^q(\Omega)),$
\item $\kappa$ stands for a generic constant that can vary from line to line,
\item $\zeta$ is the constant of the generalised Young inequality (may vary from line to line, it is chosen to be as small as necessary in the given case)
\item $a \lesssim b \Leftrightarrow a \leq \kappa b$ for some constant $\kappa$

\end{enumerate}

%%% *** THE MODEL FOR THE POLLUTED ATMOSPHERE *** %%%
\section{The model for the polluted atmosphere} \label{eqsdescribingatmpluspollution}

The key point of this new model is to integrate some of the main ideas of the Gaussian plume dispersion model into the system describing the polluted atmosphere. The cardinal assumption these dispersion models typically use is that advection is more dominant than diffusion along the downwind direction. In order to mathematically grasp these concepts and build them into our model, we begin by fixing the local Cartesian coordinate system to be adjusted to the downwind direction, i.e. $z$ is oriented upwards, while $x$ is the downwind, and $y$ is the crosswind direction.

The next step is to obtain a set of equations that describe the pollution phenomenon in the atmosphere, specifically including the information that the domain we work on is shallow.

Here the original $\Omega^\epsilon$ domain we consider is a local slice of the atmosphere written in Cartesian coordinates:

\begin{equation}\label{original_domain_eps_including}
\Omega^\epsilon = \{(x,y,z) \in \mathbb{R}^3 ; (x,y) \in \Omega_2, 0 < z < \epsilon \cdot a \},
\end{equation}
where $\Omega_2$ is a Lipschitz-domain in $\mathbb{R}^2$, and $a$ is a positive constant. We highlight that we are working on a simplified domain where a flat terrain is assumed below the atmosphere; in other words -- as $a \in \mathbb{R}^{+}$ implies -- for simplicity we ignore hills and slopes on the ground.

After setting the domain, we start with the original set of equations

\begin{equation} \label{originalEQ-V}
\partial_t \mathbf{v} + (\mathbf{v} \cdot \nabla)\mathbf{v} -\Delta_{\nu} \mathbf{v} + \nabla q + 2 \mathbf{\omega} \times \mathbf{v}= \mathbf{g} \text{ in } \Omega^\epsilon \times (0,T)
\end{equation}
\begin{equation} \label{originalEQ-I}
\nabla \cdot \mathbf{v} = 0 \text{ in } \Omega^\epsilon \times (0,T)
\end{equation}
\begin{equation} \label{diffusive}
\partial_t P + \mathbf{v} \cdot \nabla P = \nabla \cdot (\mathbf{K} \nabla P) + Q \text{ in } \Omega^\epsilon \times (0,T),
\end{equation}
where the fluid velocity is denoted by $\mathbf{v}$, the function $P$ represents the pollutant concentration and $q$ the pressure.

Now, for each space coordinate in the domain given by (\ref{original_domain_eps_including}) and all variables in equations (\ref{originalEQ-V}) - (\ref{diffusive}) we introduce a set of scaling equations. These allow us to pass from the $\epsilon-$ dependent $\Omega^\epsilon$ original domain to an $\epsilon-$free $\Omega$ domain. On the other hand, the $\epsilon$ shallowness parameter will be explicitly present in the momentum and convection-diffusion equations themselves.

The complete set of rescaling equations are a result of two different steps motivated by two separate aspects of the model: the shallowness of the domain and the concept of downwind.

Firstly, we consider the consequences of shallowness. Each of the scaling equations defined below can be heuristically deduced from the typical physical dimension of the respective quantities. Specifically, we remind that viscosity constants have physical dimensions $[\nu_{x_i}] = \text{typical length}_{x_i}^2 / \text{typical time},$ and the same holds for diffusivity (see also \cite{azerad2001mathematical}, \cite{besson1992some}, \cite{pollatmws2019}). In detail, we use the following scaling identities to arrive to the rescaled model.

\begin{equation} \label{shallownessscalingeqs}
\begin{split}
& x = x_1, \quad y = x_2, \quad z = \epsilon x_3, \\
& v_x = u_1^{\epsilon}, \quad v_y = u_2^{\epsilon}, \quad v_z = \epsilon u_3^{\epsilon}, \\
& \nu_x = \nu_1, \quad \nu_y = \nu_2 , \quad \nu_z = \epsilon ^ 2\nu_3, \\
& K_y = K_2 , \quad K_z = \epsilon^2 K_3, \\
& P = \frac{1}{\epsilon} c^\epsilon, \quad Q = \frac{1}{\epsilon} s^\epsilon, \quad q = p^{\epsilon}
\end{split}
\end{equation}

Strictly speaking, here we also use $K_x = K_1,$ but as the second step of the rescaling process focuses on this parameter, we keep the $K_x$ notation in the first step.

Now we are ready to take into account the effect of the coordinate choice itself. As diffusion is negligible compared to advection effects in the downwind direction, we have

\begin{equation} \abs{u_1 \partial_{1} c} \gg | K_x \partial_{11} ^2 c | ,\end{equation}
which leads us to introducing the new scaling identity

\begin{equation} \label{ksepsk1} K_x = \epsilon K_1. \end{equation}

\begin{remark}
\normalfont Rigorously speaking, technically we have two different small parameters in our model: the shallowness of the domain ($z = \epsilon_1 x_3$) and diffusion effects in the downwind direction ($K_x = \epsilon_2 K_1$). For simplicity we assume that these two vanishing parameters are the same, we take $\epsilon = \epsilon_1 = \epsilon_2.$
\end{remark}

Note that as an immediate result, we will have an extra $\epsilon$ term in the concentration equation's Laplacian.

Once we apply (\ref{shallownessscalingeqs}) combined with (\ref{ksepsk1}), we pass from the original $\Omega^\epsilon$ domain to the $\epsilon-$ independent

\begin{equation} \Omega = \{(x_1,x_2,x_3) \in \mathbb{R}^3 ; (x_1,x_2) \in \Omega_2, 0 < x_3 < a \} \end{equation}
set.

\begin{remark} \normalfont
The inclusion of these new features related to the downwind-matching coordinate system in a pollution model are legitimate when the following circumstances hold:
\begin{itemize}
\item We are considering a time interval which is not "too long" --- meaning that it is reasonable to assume that we have a permanent, relatively stable main wind direction, i.e. fixing the horizontal coordinates makes sense in the first place.
\item There is an adequate, relatively high wind speed --- low wind speeds in general lead to a situation where three dimensional diffusion dominates, leading to an error caused by the dominant advection assumption. Fast changing wind direction and low windspeed resulting in fast pollutant settling can make it extremely challenging to obtain reliable results from Gaussian models.
\end{itemize}
\end{remark}

Choosing the coordinate system in this particular way results in changes in three main points of the system with respect to the case of the rotating geophysical frame. One of these is the Coriolis term, since now the $x$ axis is not necessarily perpendicular to the rotation vector. We need to revisit the boundary conditions as well, and as discussed above, a modification in the pollution concentration equation also has to be taken into account. We will describe these changes and their consequences in detail --- the main challenge in proving this paper's convergence results is in fact to handle the effect of the latter change in the inequalities that provide boundedness for the required quantities.

\begin{remark} \normalfont
For clarity, we add a comment here regarding the intuitive meaning of what we understand by the downwind direction $\mathbf{x}$ (which gives orientation to a coordinate axes), and the velocity $\mathbf{v}$ itself (which is the solution function). The downwind direction is an approximately steady-in-time direction, it can be seen as the direction marked by the wind direction trackers on an airport or meteorological point. When we consider the downwind, we neglect the small perturbations in its orientation, it can be viewed as a time-averaged direction as long as the fluctuation is relatively small. On the other hand, the velocity itself is a three-dimensional, time dependent vector that changes second after second.
\end{remark}

To expand the $2 \mathbf{\omega} \times \mathbf{v}$ Coriolis term, firstly it is necessary to understand the structure of the angular velocity vector expressed in this specific coordinate system. We remind that in classical geophysical coordinates $\omega = f (0, \cos(\phi), \sin(\phi)),$ where $\phi$ denotes the latitude. Note that the downwind-matching coordinate system and the classical east-north-upwards oriented geophysical coordinate system are different only in a rotation around the $z$ axis. Let $\theta$ denote the angle between the respective $x$ axes, then the rotation vector expressed in the downwind-matching coordinate system becomes

\begin{gather}
\begin{bmatrix} \cos(\theta) & -\sin(\theta) & 0 \\ \sin(\theta) & \cos(\theta) & 0 \\ 0 & 0 & 1 \end{bmatrix} \begin{bmatrix} 0 \\ f \cos(\phi) \\ f \sin(\phi) \end{bmatrix}
=
f \begin{bmatrix} -\sin(\theta)\cos(\phi) \\ \cos(\theta) \cos(\phi) \\ \sin(\phi) \end{bmatrix}.
\end{gather}

Now, let

$$\alpha = -2 f \sin(\theta)\cos(\phi), \quad \beta = 2 f \cos(\theta) \cos(\phi), \quad \gamma = 2 f \sin(\phi),$$
i.e. $2 \omega = (\alpha, \beta, \gamma),$ then using this, the Coriolis force takes the form

\begin{gather}
\begin{bmatrix} \alpha \\ \beta \\ \gamma \end{bmatrix}
\times
\begin{bmatrix} u_1^\epsilon \\ u_2^\epsilon \\ \epsilon u_3^\epsilon \end{bmatrix}
=
\begin{bmatrix} \beta \epsilon u_3^\epsilon - \gamma u_2^\epsilon \\ \gamma u_1^\epsilon - \alpha \epsilon u_3^\epsilon \\ \alpha u_2^\epsilon - \beta u_1^\epsilon \end{bmatrix}.
\end{gather}

Thus, these are the terms that appear in the three velocity equations below; but note that in the rescaled version of the vertical velocity equation the third component of the Coriolis term is present with an $\epsilon$ multiplier, which is obtained from the $\epsilon$ term in the denominator of $\partial_z q.$

Using the expanded form of the Coriolis terms and applying the scaling process ($\ref{shallownessscalingeqs}$) combined with $(\ref{ksepsk1}),$ we arrive to the merged anisotropic equations of the polluted atmosphere:
\begin{equation} \label{ANSC1}
\partial_t u^\epsilon _1 + \mathbf{u}^\epsilon \cdot \nabla u^\epsilon _1 -\Delta_\nu u^\epsilon _1 -\gamma u_2^\epsilon + \epsilon \beta u_3^\epsilon + \partial_1 p^\epsilon = 0 \text{ in } \Omega \times (0,T)
\end{equation}
\begin{equation} \label{ANSC2}
\partial_t u_2^\epsilon + \mathbf{u}^\epsilon \cdot \nabla u_2^\epsilon -\Delta_\nu u_2^\epsilon + \gamma u_1^\epsilon - \alpha \epsilon u^\epsilon_3 + \partial_2 p^\epsilon = 0 \text{ in } \Omega \times (0,T)
\end{equation}
\begin{equation} \label{ANSC3}
\epsilon^2 \{ \partial_t u_3^\epsilon + \mathbf{u}^\epsilon \cdot \nabla u_3^\epsilon -\Delta_\nu u_3^{\epsilon} \} + \epsilon \alpha u^\epsilon_2 - \epsilon \beta u_1^\epsilon + \partial_3 p^\epsilon = 0 \text{ in } \Omega \times (0,T)
\end{equation}
\begin{equation} \label{ANSC4}
\nabla \cdot \mathbf{u}^\epsilon= 0 \text{ in } \Omega \times (0,T)
\end{equation}
\begin{equation} \label{ANSC5}
\partial_t c^\epsilon + \mathbf{u^\epsilon} \cdot \nabla c^\epsilon = \epsilon K_1 \partial_{11} ^ 2 c^\epsilon + K_2 \partial_{22} ^ 2 c^\epsilon + K_3 \partial_{33} ^ 2 c^\epsilon + s^\epsilon \text{ in } \Omega \times (0,T).
\end{equation}

The initial condition for the velocity and the pollution concentration are
\begin{equation} \label{ANSC_init}
\mathbf{u}^\epsilon (\cdot, t = 0) = \mathbf{u}_0, \quad c^\epsilon (\cdot, t=0) = c_0 \quad \text{ in } \Omega.
\end{equation}

We will specify the regularity of the initial data in Sections \ref{sectionWS} and \ref{sectionSTRONGsol} according to the fact whether we are dealing with weak or strong solutions.

This system will become complete when we define the boundary conditions --- we describe these in detail in Sections \ref{sectionWS} and \ref{sectionSTRONGsol} as they are chosen with different considerations depending on the type of the solutions we are dealing with.

If we assume $\mathbf{u}^\epsilon = O(1)$, then neglecting the $\epsilon^2$ and $\epsilon$ including terms in the anisotropic equations (\ref{ANSC1}) - (\ref{ANSC5}), we formally arrive to the following hydrostatic Navier-Stokes equations (i.e. primitive equations) combined with pollution:

\begin{equation} \label{HNSC1}
\partial_t u_1 + \mathbf{u} \cdot \nabla u_1 - \Delta_\nu u_1 -\gamma u_2 + \partial_1 p = 0 \text{ in } \Omega \times (0,T)
\end{equation}
\begin{equation} \label{HNSC2}
\partial_t u_2 + \mathbf{u} \cdot \nabla u_2 - \Delta_\nu u_2 + \gamma u_1 + \partial_2 p = 0 \text{ in } \Omega \times (0,T)
\end{equation}
\begin{equation} \label{HNSC3}
\partial_3 p = 0 \text{ in } \Omega \times (0,T)
\end{equation}
\begin{equation} \label{HNSC4}
\nabla \cdot \mathbf{u} = 0 \text{ in } \Omega \times (0,T)
\end{equation}
\begin{equation} \label{HNSC5}
\partial_t c + \mathbf{u} \cdot \nabla c = K_2 \partial_{22} ^ 2 c + K_3 \partial_{33} ^2 c + s \text{ in } \Omega \times (0,T)
\end{equation}
\begin{equation} \label{HNSC_init}
u_i (\cdot, t=0) = u_{0 i}, \quad c (\cdot, t=0) = c_0 \text{ in } \Omega, \text{$i$ = 1,2}
\end{equation}

The specific definitions of the $s^\epsilon$ and $s$ source terms also depend on whether we consider weak or strong solutions. In both cases they are fundamentally based on the notion of the approximated delta function with parameter $\epsilon$ and the Dirac delta distribution, but their exact definitions will be given in the subsections according to the regularity of the solutions.

Working with this setup in the following sections, in certain computations we will deal with the velocity function separately from the concentration solution, hence here we make a technical assumption, namely that
\begin{equation}\label{u_independent_from_c}\text{air velocity is not changed by the presence of contaminants.}\end{equation}
This also means that the velocity of air particles and the velocity of pollutants can be different. Taking into account that the diffusion of aerosols or pollens does not generate significant atmospheric air flows, assumption (\ref{u_independent_from_c}) is a reasonable simplification, and thus we will be able to consider $\mathbf{u}^\epsilon$ and $\mathbf{u}$ without the respective concentration functions.

To conclude the description of the model we highlight an important feature of the system that is suggested also by the structure of the governing equations (\ref{ANSC5}) and (\ref{HNSC5}). With respect to pollution diffusion, directions $y$ and $z$ have a particular and shared importance compared to the $x$ direction, namely that we have diffusion taking place on $\mathcal{O}(1)$ scale in the former ones. 
%This suggests the introduction of the generalised direction $d$ which includes directions $y$ and $z$ (similarly to the more standard concept of $h$ standing for the horizontal directions $x$ and $y$). Furthermore, we accordingly introduce the notations $\Delta_d = K_2 \partial_{22}^2 + K_3 \partial_{33}^2$ and $\nabla_d = (\partial_2, \partial_3).$

The following two sections are dedicated to state and prove a rigorous connection on the level of solutions between the anisotropic and hydrostatic models described above. Section \ref{sectionWS} focuses on the verification of a convergence result between the models (\ref{ANSC1}) - (\ref{ANSC_init}) and (\ref{HNSC1}) - (\ref{HNSC_init}) in the framework of weak solutions. In Section \ref{sectionSTRONGsol} we apply stronger initial regularities and -- from the computations' point of view -- a more favourable boundary structure: here we describe an analogous result to that in Section \ref{sectionWS}, but the statement holds for strong solutions.

%%% WEAK SOLUTIONS %%%
\section{Weak solutions framework} \label{sectionWS}

This section is dedicated to extend the weak convergence result of \cite{pollatmws2019} to the downwind-matching coordinate system in the framework of weak solutions. We begin by making the modelling equations complete by defining the boundary conditions. In the scenario of weak solutions we aim to be as close as possible to the original domain, thus we consider a classical non-periodic local domain on the surface of the Earth with minimal modifications due to the new $\epsilon$ term.

Here the boundary structure of the domain is identical to what we used in \cite{pollatmws2019}, the only difference is made by the kind of boundary conditions we allow on the lower boundary: for the case of the downwind-matching coordinates here we restrict ourselves to the technically simpler inland scenario (i.e. where any significant nearby water body is excluded). The upper boundary of the $\Omega$ domain is denoted by $\Gamma_U$, the lateral boundary is $\Gamma_L$, the notation $\Gamma_A$ stands for the $\Gamma_L \cup \Gamma_U$ set (above the ground), and the lower boundary of the domain is $\Gamma_G$, while the complete boundary is denoted by $\Gamma$ (see Figure $\ref{structdomaingraphboundaries_weak}$). For simplicity we assume that the altitude of $\Omega$ is $1$, thus we can also express the full domain as $\Omega = \Gamma_G \times [0,1]$.
\begin{figure}[!h]
\centering
\begin{tikzpicture}
\fill[color=blue!10] (2,0) -- (2,3) arc (360:180:2cm and 0.5cm) -- (-2,0) arc (180:360:2cm and 0.5cm);
\fill[color=blue!5] (0,3) circle (2cm and 0.5cm);

\draw[thick] (-2,3) -- (-2,0) arc (180:360:2cm and 0.5cm) -- (2,3) ++ (-2,0) circle (2cm and 0.5cm);
\draw[densely dashed, thick] (-2,0) arc (180:0:2cm and 0.5cm);
\draw (-2,1) -- (-3,1) -- (-5,-1) -- (3,-1) -- (5,1) -- (2,1);
\draw[dashed] (-2,1) -- (2,1);

\fill[green!10,opacity=0.5] (-2,0) -- (-2,1) -- (-3,1) -- (-5,-1) -- (3,-1) -- (5,1) -- (2,1) -- (2,0) arc (0:180:2cm and -0.5cm);
\fill[color=green!5] (0,0) circle (2cm and 0.5cm);

\draw (0,1.5) node {$\Omega$};
\draw (0,3) node {$\Gamma_U$};
\draw (0,0) node {$\Gamma_G$};
\draw (-2.3,1.5) node {$\Gamma_L$};
\draw (2.3,1.5) node {$\Gamma_L$};

\draw[thick] (2,0) arc (360:180:2cm and 0.5cm);
\end{tikzpicture}
\caption{The boundary structure of $\Omega$ for the case of weak solutions}
\label{structdomaingraphboundaries_weak}
\end{figure}
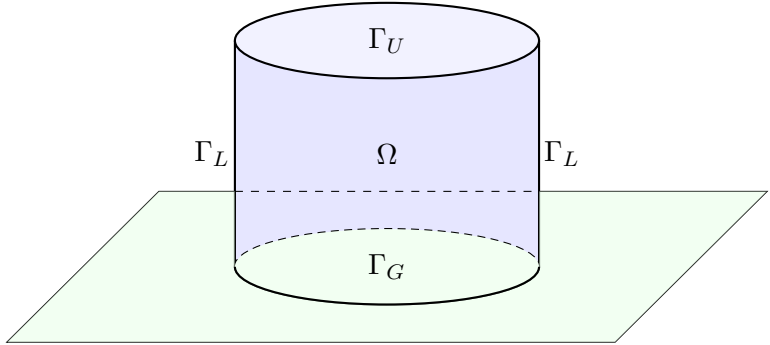

The boundary conditions to describe the anisotropic system (\ref{ANSC1}) - (\ref{ANSC_init}) can be summarised as

\begin{equation} \label{ANSC_bc-weak}
\begin{split}
& \mathbf{u}^\epsilon = 0 \text{ on } \Gamma \times (0,T), \\
& c^\epsilon = 0 \text{ on } \Gamma_A \times (0,T), \\
& \partial_3 c^\epsilon = 0 \text{ on } \Gamma_G \times (0,T), \\
\end{split}
\end{equation}
while for the hydrostatic case we use
\begin{equation} \label{HNSC_bc-weak}
\begin{split} 
& \mathbf{u} = 0 \text{ on } \Gamma \times (0,T), \\
& c = 0 \text{ on } \Gamma_A \times (0,T), \\
& \partial_3 c = 0 \text{ on } \Gamma_G \times (0,T). \\
\end{split}
\end{equation}

\begin{remark}
\normalfont Compared to \cite{pollatmws2019}, we highlight two important differences concerning the definition of the boundary conditions at ground level. Firstly, since now we are working exclusively above inland surfaces, we have modified the boundary condition for $\mathbf{u}_h$ on $\Gamma_G.$ In more detail, the framework of our previous paper was built around the classical geophysical coordinate system -- the analytical advantages of this core setting in \cite{pollatmws2019} allowed us to apply the relatively generic $\nu_3 \partial_3 \mathbf{u}_h = \theta_h$ Neumann boundary condition at ground level. In this new, current scenario however, we need to adapt to the mathematical challenges induced by the downwind-matching coordinate system, and hence we need a stricter boundary condition on $\Gamma_G.$ Specifically, we apply $\mathbf{u}=0$ not only above the surface of the Earth, but at ground level too --- this will be crucial later because it allows us to use the Poincar\'e inequality.
Concerning the concentration function, note that the condition $\mathbf{M} \nabla c \cdot \vec{\boldsymbol{n}}_{\Gamma_G} = 0$ reduces exactly to $\partial_3 c = 0$ on $\Gamma_G$ for the case of a diagonal diffusion matrix $\mathbf{M}$, which we have previously assumed to use here.
\end{remark}

As we have mentioned it in the Introduction, in this case we will work with a more regular ${\mathbf{u}_h}_0 \in H^1$ initial data. Indeed concerning the initial data, we set

\begin{equation}
\begin{gathered}\label{DWMC_weak_initial_data_regularities}
{\mathbf{u}_h}_0 \in H^1 (\Omega), {u_3}_0 \in L^2(\Omega) \text{ with } \nabla \cdot \mathbf{u}_0 = 0,\\
{u_3}_0 = - \int_0^z \nabla_h \cdot {\mathbf{u}_h}_0 \diff z' \text{ and } c_0 \in L^2 (\Omega).
\end{gathered}
\end{equation}

We emphasise here that we are still working in classical Leray framework of weak solutions, even though we use an $H^1(\Omega)$ regularity assumption on ${\mathbf{u}_h}_0,$ which is stronger than the more usual $L^2$ space regularity requirement for weak solutions (\cite{azerad2001mathematical}, \cite{pollatmws2019}), because of \eqref{DWMC_weak_initial_data_regularities}, we have that  ${u_3}_0$ is merely in  $L^2$. The reason we need this stronger regularity is to counterbalance the decreased control we have on the concentration function due to the coordinate choice, but overall this does not lead us out of the field of weak solutions.

Finally, the $s^\epsilon$ and $s$ source terms are formulated in the spirit of our previous result \cite{pollatmws2019}: the hydrostatic $s$ function is defined as a point source, while the anisotropic source function $s^\epsilon$ is modelled with a parametrised approximated delta function. For simplicity we assume here that the pollution process begins at $t = 0,$ furthermore without loss of generality we assume that the source is centred in $x=0$.
Let $\delta_\epsilon$ be the approximation of the Dirac $\delta$ distribution given by 
\begin{equation} \label{delta_eps_source_def}
\delta_\epsilon = 
\begin{cases}
\frac{1}{\epsilon^3} \exp{\bigg ( \frac{1}{\abs{{\frac{x}{\epsilon}}}^2-1} \bigg ) } &\quad \abs{x} < \epsilon \\
0 &\quad \abs{x} \geq \epsilon. \\
\end{cases}
\end{equation}

Following the main ideas described above, we set

\begin{equation}
s^\epsilon = \delta_\epsilon, \quad s = \delta.
\end{equation}

%%% *** WEAK FORMULATION AND MAIN RESULT*** %%%
\subsection{Weak Formulation and Main Result} \label{weakformulationmainresult}

In this section we describe the weak formulation of the merged equations in both the anisotropic and hydrostatic case, and we state our main result.

We need to define the following spaces:

$$C^{\infty}_\Gamma (\Omega) = \{ \varphi \in C^\infty (\bar{\Omega}); \varphi = 0 \text{ in some neighbourhood of } \Gamma\},$$
%\[ C^{\infty}_\Gamma (\Omega) = \{ \varphi \in C^\infty (\bar{\Omega}); \partial^\alpha \varphi = 0 \text{ in some neighbourhood of } \Gamma \text{ for any } \alpha \in \mathbb{N} \} \]
%
\begin{align*}
H^k_{\Gamma} (\Omega) = & \overline{C_\Gamma ^ \infty (\Omega)} ^ {H^k (\Omega)} = \{ v \in H^k(\Omega);\  \partial^\alpha v=0 \ \text{on $\Gamma$ for any  $|\alpha| < k$} \},
\end{align*}
%\begin{align}
%H^k_{\Gamma} (\Omega) = & \overline{C_\Gamma ^ \infty (\Omega)} ^ {H^k (\Omega)} = \{ v \in H^k(\Omega); \nonumber \\
%& \partial^\alpha v=0 \text{ in some neighbourhood of } \Gamma \text{ for any } |\alpha| < k\nonumber \},
%\end{align}
%
%\begin{align}
%H^k_{\Gamma_A} (\Omega) = & \{ v \in H^k(\Omega); \nonumber \\
%& \partial^\alpha v=0 \text{ in some neighbourhood of } \Gamma_A \text{ for any } |\alpha| < k \nonumber \},
%\end{align}
\begin{align*}
H^k_{\Gamma_A} (\Omega) = & \{ v \in H^k(\Omega); \  \partial^\alpha v=0 \ \text{on $\Gamma_{A}$ for any  $|\alpha| < k$} \},
\end{align*}
\begin{align}
L^2_{\Gamma_A}(\Omega) = & \{ c \in L^2(\Omega); c = 0 \text{ on } \Gamma_A \text{ in the sense of } \langle c , \tilde{c} \rangle_{H^{-1/2} (\Gamma_A) \times H^{1/2} (\Gamma_A)} = 0 \nonumber \\
& \text{ for any } \tilde{c} \in H^{1/2}(\Gamma_A) \nonumber \},
\end{align}

$$ \mathbf{V} = \{ \mathbf{v} \in H^1_{\Gamma} (\Omega) \times H^1_{\Gamma} (\Omega) \times H^1_{\Gamma} (\Omega); \nabla \cdot \mathbf{v} = 0 \text{ in } \Omega \}, $$

$$\mathbf{W} = \{ \mathbf{v} \in (H^1_{\Gamma} (\Omega) \cap H^2(\Omega)) \times (H^1_{\Gamma} (\Omega) \cap H^2(\Omega)) \times H^1_{\Gamma} (\Omega); \nabla \cdot \mathbf{v} = 0 \text{ in } \Omega \}. $$
Furthermore we introduce the notation $b(\mathbf{u}_h)$ = $\gamma ( - u_2 , u_1).$

Then the weak formulation of the hydrostatic system (\ref{HNSC1}) - (\ref{HNSC_init}), (\ref{HNSC_bc-weak}) takes the following form (see also \cite{cao2007global} for example).

\begin{definition} \label{definitionhydrostaticWF}

The pair $(\mathbf{u}, c)$ is called a weak solution of the a hydrostatic system (\ref{HNSC1}) - (\ref{HNSC5}) subject to (\ref{HNSC_init}), (\ref{HNSC_bc-weak}) if $\mathbf{u} = (\mathbf{u}_h, u_3) \in L^2 (0, T; \mathbf{W}) \cap L^\infty (0,T; L^2(\Omega)^3),$ $\mathbf{u}_h \in L^\infty(0,T; H^1(\Omega)),$ and $c \in L^\infty (0, T; L_{\Gamma_A}^2(\Omega))$ with $\partial_2 c$, $\partial_3 c \in L^2(0, T; L^2(\Omega)),$ moreover $\mathbf{u}$ and $c$ satisfy the integral identities

\begin{equation}
\begin{split}
& \int_0^T \bigg[ -(\mathbf{u}_h, \partial_t \mathbf{\tilde{u}}_h) + (\nabla_\nu \mathbf{u}_h, \nabla_\nu \mathbf{\tilde{u}}_h) - (\mathbf{u}_h, (\mathbf{u} \cdot \nabla) \mathbf{\tilde{u}}_h) + (b(\mathbf{u}_h), \mathbf{\tilde{u}}_h) \bigg ] \diff t \\
& = - ({\mathbf{u}_h}_0, \mathbf{\tilde{u}}_h(0)) \label{hydrostaticWF_u}
\end{split}
\end{equation}
and
\begin{equation}
\begin{split} 
&+ \int_0^T \bigg[ -(c, \partial_t \tilde{c}) -(\mathbf{u}c, \nabla \tilde{c}) + K_2 (\partial_2 c, \partial_2 \tilde{c}) + K_3 (\partial_3 c, \partial_3 \tilde{c}) \bigg ] \diff t \\
& = - (c_0, \tilde{c}(0)) + \int_0^T (s,\tilde{c}) \diff t \label{hydrostaticWF_c}
\end{split}
\end{equation}
for all $(\mathbf{\tilde{u}}, \tilde{c})$ with $\mathbf{\tilde{u}} = (\mathbf{\tilde{u}}_h, \tilde{u}_3) \in H^1 (0,T, \mathbf{W})$, with $\mathbf{\tilde{u}}_h (T) = 0$ and $\tilde{c} \in H^1 (0,T, H^2_{\Gamma_A} (\Omega))$, $\tilde{c} (T) = 0$.

\end{definition}

Giving definition to the weak formulation of the (\ref{ANSC1}) - (\ref{ANSC_init}), (\ref{ANSC_bc-weak}) anisotro-pic system we underline a feature regarding the spaces we use here. The velocity part of the solution is set to be in $C_w$ since it corresponds to the requirements of the proof detailed later. Even though we do not need the same weak continuity in time for the concentration function, this follows as a consequence once the velocity is a Leray-Hopf weak solution. Thus, overall we work with Leray-Hopf weak solutions, but the spaces in the below definition are kept minimal (in other words we use $L^\infty$ regularity instead of $C_w$ regularity where it is sufficient).

\begin{definition} \label{definitionanisotropicWF}

The pair $(\mathbf{u}^\epsilon, c^\epsilon)$ is called a weak solution of the anisotropic system (\ref{ANSC1}) - (\ref{ANSC5}) subject to (\ref{ANSC_init}), (\ref{ANSC_bc-weak}) if $\mathbf{u}^\epsilon = (\mathbf{u}^\epsilon_h, u^\epsilon_3) \in L^2 (0,T ; \mathbf{V}) \cap C_w (0, T; L^2 (\Omega)^3)$ and $c^\epsilon \in L^\infty (0,T; L^2(\Omega)) \cap L^2 (0,T; H_{\Gamma_A}^1(\Omega))$, moreover $\mathbf{u}^\epsilon$ and $c^\epsilon$ satisfy the integral identities

\begin{equation}
\begin{split}
& \int_0^T \bigg [ -(\mathbf{u}^\epsilon_h, \partial_t \mathbf{\tilde{u}}_h) + (\nabla_\nu \mathbf{u}^\epsilon_h, \nabla_\nu \mathbf{\tilde{u}}_h) - (\mathbf{u}^\epsilon_h, (\mathbf{u}^\epsilon \cdot \nabla) \mathbf{\tilde{u}}_h) + (b(\mathbf{u}^\epsilon_h), \mathbf{\tilde{u}}_h) \bigg ] \diff t \\
& + \epsilon \int_0 ^ T \big [ (\beta u_3 ^\epsilon, \tilde{u}_1) - (\beta u_1^\epsilon, \tilde{u}_3) \big ] \diff t + \epsilon \int_0 ^ T \big [ (\alpha u_2 ^\epsilon, \tilde{u}_3) - (\alpha u_3^\epsilon, \tilde{u}_2) \big ] \diff t \\
& + \epsilon^2 \int _0 ^ T \big [ -(u_3^\epsilon, \partial_t \tilde{u}_3) +(\mathbf{u}^\epsilon \cdot \nabla u_3 ^ \epsilon, \tilde{u}_3) + (\nabla_\nu u_3^\epsilon, \nabla_\nu \tilde{u}_3) \big ] \diff t \\
& = - ({\mathbf{u}^\epsilon_h}_0, \mathbf{\tilde{u}}_h(0)) - \epsilon^2 ({u^\epsilon_3}_0, \tilde{u}_3 (0)) \raisetag{20pt} \label{anisotropicWF_u}
\end{split}
\end{equation}
and
\begin{equation}
\begin{split} 
& \int_0^T \bigg [ -(c^\epsilon, \partial_t \tilde{c}) -(\mathbf{u}^\epsilon c^\epsilon, \nabla \tilde{c}) + \epsilon K_1 (\partial_1 c^\epsilon, \partial_1 \tilde{c}) \\
& + K_2 (\partial_2 c^\epsilon, \partial_2 \tilde{c}) + K_3 (\partial_3 c^\epsilon, \partial_3 \tilde{c}) \bigg ] \diff t \\
& = - (c_0, \tilde{c} (0)) + \int_0^T (s^\epsilon,\tilde{c}) \diff t \label{anisotropicWF_c}
\end{split}
\end{equation}
for all $\mathbf{\tilde{u}} = (\mathbf{\tilde{u}}_h, \tilde{u}_3) \in H^1 (0,T;\mathbf{V})$ with $\mathbf{\tilde{u}}(T) = 0$ and $\tilde{c}$ such that $\tilde{c} \in H^1 (0,T, H^3_{\Gamma_A} (\Omega))$ and $\tilde{c} (T) = 0$.

\end{definition}

\begin{remark}
\normalfont Note that in the above definitions the spaces we set for the test functions are not rigorously minimal: our choice is motivated by simplicity. For completeness we highlight that for the velocity test function in Definition \ref{definitionhydrostaticWF} it is sufficient to require that $\mathbf{\tilde{u}} \in L^2 (0,T, \mathbf{W}) \cap H^1(0,T, L^2(\Omega)).$ Similarly, in Definition \ref{definitionanisotropicWF} we can reduce the space applied for $\mathbf{\tilde{u}}$ from $H^1 (0,T;\mathbf{V})$ to the stricter $L^2 (0,T;\mathbf{V}) \cap H^1 (0,T; L^2(\Omega))$ set. The analogous statements hold for the concentration test functions as well.
\end{remark}

Now we are ready to state the main theorem of this section.

\begin{theorem} \label{maintheoremWEAK}
Let ${\mathbf{u}_h}_0$ $\in$ $H^1 (\Omega)$, ${u_3}_0 \in L^2(\Omega)$ with $\nabla \cdot \mathbf{u}_0 = 0, {\mathbf{u}_h}_0 = 0, {u_3}_0 n_3 = 0$ on $\partial \Omega$; let $c_0 \in$ $L^2 (\Omega)$, and assume that the boundary conditions (\ref{ANSC_bc-weak}) hold; then as the aspect ratio $\epsilon$ tends to zero, any weak solution $(\mathbf{u}^\epsilon, c^\epsilon)$ of the anisotropic equations (\ref{ANSC1}) - (\ref{ANSC_init}) converges (up to a subsequence) weakly to a weak solution $(\mathbf{u}, c)$ of the hydrostatic equations of the polluted atmosphere (\ref{HNSC1}) - (\ref{HNSC_init}).
\end{theorem}

The statement of the main theorem of this first part of the paper is analogous to our previous result stated for the case of the geophysical frame, but we highlight the difference that here we use the initial data ${\mathbf{u}_h}_0 \in H^1,$ while for the case of the geophysical frame an $L^2$ initial velocity was sufficient.

%%% *** PROOF *** %%%
\subsection{Proof of the main theorem} \label{proof}

This section is dedicated to the verification of Theorem \ref{maintheoremWEAK}. Structurally we follow a similar chain of thought as in our previous paper until the point where the proof diverges reaching the weak convergence of the nonlinear term $u_3^\epsilon c^\epsilon,$ where in this case we can not benefit from an $H^1$-regular concentration function. In more detail, we begin by deriving the energy inequality and collect the a priori estimates that derive from it, then verify the less problematic weak convergence results. Finally we elaborate the steps that bring the convergence of the $u_3^\epsilon c^\epsilon$ term to a closure. The main idea here is to temporarily ignore the $c^\epsilon$ function and prove a strong convergence result considering only the velocity functions. Throughout this step we go beyond the weak solutions quality and show that we are working with strong solutions for the velocity equation. In this new setting with more favourable regularities we can prove the convergence result we need to deal with the nonlinear terms involving $u_3^\epsilon$ and  $c^\epsilon$. Finally as a byproduct we get a convergence rate for the velocity (see Proposition \ref{convrateweakframework}).

\subsubsection{Energy inequality and a priori estimates}

Deriving the energy inequality for the anisotropic system (\ref{ANSC1}) - (\ref{ANSC_init}), (\ref{ANSC_bc-weak}), we obtain that for a.e. $t \in [0,T]$ the following holds:

\begin{equation}
\begin{split}
& \frac{1}{2} (\norm{\mathbf{u}_h^\epsilon (t)}_{L^2}^2 + \epsilon ^ 2 \norm{u_3^\epsilon (t)}_{L^2}^2 + \norm{c^\epsilon (t)}_{L^2}^2) + \int_0 ^ t \big[ \norm{\nabla_\nu \mathbf{u}_h^\epsilon}_{L^2}^2 + \epsilon ^ 2 \norm{\nabla_\nu u_3^\epsilon}_{L^2}^2 \big] \diff \tau \\
& + \int_0 ^ t \big [ \epsilon K_1 \norm{\partial_1 c^\epsilon}_{L^2}^2 + K_2 \norm{\partial_2 c^\epsilon}_{L^2}^2 + K_3 \norm{\partial_3 c^\epsilon}_{L^2}^2 \big ] d \tau \\
& \leq \int_0^t (s^\epsilon,c^\epsilon) \diff \tau + \frac{1}{2} (\norm{{\mathbf{u}_h}_0}_{L^2}^2 + \epsilon ^ 2 \norm{{u_3}_0}_{L^2}^2 + \norm{c_0}_{L^2}^2), \raisetag{50pt} \label{finalformEIEQ}
\end{split}
\end{equation}
where he right-hand side  is apparently bounded.

We note that the most significant difference in this new form of the energy inequality (\ref{finalformEIEQ}) is having $\sqrt{\epsilon} \partial_1 c^\epsilon$ in the place of $\partial_1 c^\epsilon,$ which is what we had for the case of the geophysical coordinate system. It leads to the disappearance of the $H^1$ regularity of $c^\epsilon,$ and thus we don't have the strong convergence feature of the concentration function anymore.

From  the energy inequality (\ref{finalformEIEQ}) we obtain uniform a priori estimates that we summarise in the following proposition.

\begin{prop} \label{aprioriprop}
Let $u^\epsilon_1, u^\epsilon_2, u^\epsilon_3$ and $c^\epsilon$ be the weak solutions of the system (\ref{ANSC1}) - (\ref{ANSC5}) in the sense of Definition \ref{definitionanisotropicWF}. Then it holds,

\begin{align}
& u^\epsilon_1, u^\epsilon_2, \epsilon u^\epsilon_3, c^\epsilon & \quad & \text{are bounded in $L^\infty (0,T; L^2(\Omega)),$} \label{proposition1} \\
& u^\epsilon_3 & \quad & \text{is bounded in $L^2 (0,T; L^2(\Omega)),$} \label{propositionExtra} \\
& u^\epsilon_1, u^\epsilon_2, \epsilon u^\epsilon_3 & \quad & \text{are bounded in $L^2(0,T; H^1(\Omega)),$} \label{proposition2} \\
& \sqrt{\epsilon} \partial_1 c^\epsilon, \partial_2 c^\epsilon, \partial_3 c^\epsilon & \quad & \text{are bounded in $L^2(0,T;L^2(\Omega)).$} \label{proposition3}
\end{align}

\end{prop}

For the details of the proof see \cite{pollatmws2019}.

\subsubsection{Weak convergence - Part 1.}

The convergence of the linear terms and also the convergence of the nonlinear terms with the exception of $u_3^\epsilon c^\epsilon$ can be proved analogously as in  \cite{pollatmws2019}, the proofs of these properties are not affected by the changes induced by choosing to use the downwind-matching coordinate system. Hence, for completeness we will describe the main steps leading up to obtaining these results, but we will skip the details. For technical or specific issues see \cite{azerad2001mathematical} and \cite{pollatmws2019}.

As a consequence of Proposition \ref{aprioriprop}, up to subsequences still denoted by the same way, we have the following weak convergence results.

\begin{align} 
& \mathbf{u}_h^\epsilon \rightharpoonup \mathbf{u}_h &\quad & \text{weakly in $L_t^\infty L_x^2 \cap L_t^2 H_x^1$,} \label{WeakConvUHEps} \\
& \epsilon^2 u_3^\epsilon \to 0 &\quad & \text{strongly in $L_t^\infty L_x^2 \cap L_t^2 H_x^1$,} \label{WeakConvU3Eps} \\
& u_3^\epsilon \rightharpoonup u_3 &\quad & \text{weakly in $L_t^2 L_x^2$,} \label{WeakConvU3NoEps} \\
& c^\epsilon \rightharpoonup c &\quad & \text{weakly in $L_t^\infty L_x^2$,} \label{WeakConvCEps1} \\
& \partial_2 c^\epsilon \rightharpoonup \partial_2 c &\quad & \text{weakly in $L_t^2 L_x^2$,} \label{WeakConvCEps2} \\
& \partial_3 c^\epsilon \rightharpoonup \partial_3 c &\quad & \text{weakly in $L_t^2 L_x^2$,} \label{WeakConvCEps3} \\
& \sqrt{\epsilon} \partial_1 c^\epsilon \rightharpoonup e &\quad & \text{weakly for some limit element $e$ in $L_t^2 L_x^2$.} \label{WeakConvCEps4}
\end{align}

Now, the next step is to pass to the limit in the weak formulation (\ref{anisotropicWF_u}) - (\ref{anisotropicWF_c}). Taking the limit for the linear terms can be done in a standard way using the previously obtained results.
The convergence of the nonlinear velocity terms (i.e. nonlinear terms that do not include the concentration function), requiring the compactness of $\mathbf{u}_h^\epsilon$, 
%will be the direct consequence of our calculations detailed in the following subsection (see (\ref{horizontal_velocities_convergence}) in Proposition \ref{convrateweakframework}), observe however that for these specific terms the original argument of \cite{pollatmws2019} works too.
works with same arguments as in \cite{pollatmws2019}. 

In particular, the convergence of these velocity terms can be shown applying a generalisation of the classical translation criterium of Riesz-Fr\'echet-Kolmogorov (see Theorem 5.1 in \cite{azerad2001mathematical}) which enables us to get strong convergence for the horizontal velocities. This is achieved by establishing a bound for the perturbation of $\mathbf{u}_h^{\epsilon}$ of the form $\norm{\tau_\upsilon \mathbf{u}_h^{\epsilon} - \mathbf{u}_h^{\epsilon}}_{L^2(0,T-\upsilon, {H^2}^*)} \leq \varphi(\upsilon) + \psi(\epsilon),$ where $\varphi$ and $\psi$ are appropriate functions with their limits vanishing at zero, and $\tau_\upsilon \mathbf{u}_h^{\epsilon} = \mathbf{u}_h^{\epsilon} (t+\upsilon) $. After obtaining strong convergence for $\mathbf{u}_h,$ the proof of the convergence for these nonlinear terms can be closed by applying basic interpolation techniques and the generalised H\"older inequality. By passing into the limit in \eqref{anisotropicWF_u} we get that $\mathbf{u}_h$ satisfies the weak formulation \eqref{hydrostaticWF_u}.
Finally, as for the concentration-including nonlinear terms, we have the convergence of $(u_i c^\epsilon, \partial_i \tilde{c})$ for $i=1,2,$ since as we mentioned before, we have the strong convergence of the horizontal velocities --- while in the case of $i=3$ this property does not hold for $u_3$ and from the energy bound we do not have more than the weak convergence of $c^\epsilon$. Because of its complexity we handle this specific nonlinear compound product term in the next section.

\subsubsection{Weak convergence - Part 2.}

As we mentioned before, the cornerstone idea for obtaining the convergence of the remaining product term $u_3^\epsilon c^\epsilon$ is to exploit the $H^1$ regularity of the initial horizontal velocity. We have already proved that $\mathbf{u}_h$ satisfies \eqref{hydrostaticWF_u} therefore it is a weak solution for the velocity primitive equations. On the other hand, given an $H^1$ initial data for $\mathbf{u}_h,$ according to the theory developed in \cite{kukavica-ziane}, there exists a unique global strong solution for the primitive equations $(\ref{HNSC1})$ - $(\ref{HNSC4}).$ We recall here this result (for more details see \cite{kukavica-ziane}).

\begin{theorem} \label{PE3DWellPosedness}
Let ${\mathbf{u}_h}_0 \in H^1(\Omega).$ Then there exists a unique strong solution $\mathbf{u}_h$ of the reformulated Primitive Equations $(\ref{HNSC1})$ - $(\ref{HNSC4})$, $(\ref{HNSC_init})$ such that $$\mathbf{u}_h \in C([0, \infty), H^1(\Omega)) \cap L^2_{loc}([0, \infty), H^2(\Omega))$$ and $$\partial_t \mathbf{u}_h \in L^2_{loc} ([0,\infty), L^2(\Omega)).$$
\end{theorem}

\begin{remark} \normalfont
For the sake of completeness we highlight that the boundary conditions applied in \cite{kukavica-ziane} do not correspond perfectly to our scenario. Although both \cite{kukavica-ziane} and the present paper use analogous, Dirichlet-type boundary conditions on the lateral sections and at the bottom, our approaches diverge at the top layer (specifically, the former work applies the $\partial_3 \mathbf{u}_h = 0$ Neumann-type boundary condition, while we set $\mathbf{u}_h = 0$ on $\Gamma_U$). However, the computations carried out in \cite{kukavica-ziane} remain valid for this modified case as well. In more detail, similarly to how the authors apply appropriate limiting arguments to benefit from $\partial_3 \mathbf{u}_h$ disappearing on the bottom (n.b. in terms of boundary conditions they only have $\mathbf{u}_h = 0$ at this section), we can also justify using an analogous Neumann-type equality on the top -- and hence the respective calculations still hold.
\end{remark}

Combining Theorem \ref{PE3DWellPosedness}. with weak--strong uniqueness results for the primitive equations (in the spirit of \cite{GNT2019}), we get that $\mathbf{u}_h$ is a strong solution of (\ref{HNSC1}) - (\ref{HNSC4}).
In more detail, we have the following regularities for the strong solution of the primitive equations for the case of an $H^1$-regular horizontal initial data:

\begin{corollary} \label{StrongSolPE_boundedness}

Let $(\mathbf{u}_h,u_3)$ be the unique global strong solution to the primitive equations (\ref{HNSC1}) - (\ref{HNSC4}) with an initial data ${\mathbf{u}_h}_0 \in H^1(\Omega).$ Then we have the following boundedness result:

$$ \sup_{0 \leq t < \infty} \norm{\mathbf{u}_h}_{H^1}^2 (t) + \int_0^\infty \norm{\nabla \mathbf{u}_h}_{H^1}^2 \diff t \leq \kappa (\norm{{\mathbf{u}_h}_0}_{H^1}, \Gamma_G).$$

\end{corollary}

This is a direct consequence of Theorem \ref{PE3DWellPosedness}, but for completeness we also mention that for the simpler case of a periodic domain it has also been verified in \cite{li2019primitive}, where the proof becomes more straightforward thanks to the structure of the domain.

The result of Corollary \ref{StrongSolPE_boundedness} on the strong solutions of the primitive equations will be indispensable later, as we will use these strong solution functions for testing the scaled Navier-Stokes equations.

The following part of the proof is dedicated to verifying the strong convergence of $u_3^\epsilon,$ and is based on the idea of \cite{li2019primitive} --- our scenario is different however in that \textit{a)} first of all our domain is bounded, while the domain the authors of \cite{li2019primitive} consider is periodic, \textit{b)} we do not assume zero-averaged functions, and finally \textit{c)} in our case the Coriolis terms are incorporated to the model as well. Our strategy in providing the proof is to give a self-contained result that describes all the steps that build up the verification process, but we elaborate in more detail where the above mentioned \textit{a), b), c)} points have a tangible effect. Some technical details that are identical to those in \cite{li2019primitive} will be skipped.

Another structural difference we need to underline in connection with the proof is that instead of using the $[-1,1]$ set to represent the $\epsilon$-free vertical domain (i.e. the $\Gamma_L$ section), we use $[0,1],$ without assuming the vertical velocity function $u_3^\epsilon$ to be odd with respect to $z.$ The vertical section is commonly set to $[-1,1],$ or equivalently, to $[-a,a]$ in works that use the periodical domain scenario, see also \cite{cao2016global}, \cite{cao2017global}. In our case however we use a bounded domain without assuming a symmetry line in the vertical midpoint --- thus we work with the $[0,1]$ set as the vertical domain. As a consequence, the Ladyzhenskaya-type inequality presented in different versions in \cite{cao2003global} and \cite{li2019primitive} will be used here in the following form:

\begin{lemma}\label{LadyzhenskayatypeIEQ}

The inequalities

\begin{equation}\nonumber
\begin{split}
& \int_{\Gamma_G} \bigg ( \int_0^1 f(x,y,z) \diff z \bigg) \bigg ( \int_0^1 g(x,y,z) h(x,y,z) \diff z \bigg ) \diff x \diff y \\
& \lesssim \norm{f}_2^{1/2} \big ( \norm{f}_2^{1/2} + \norm{\nabla_h f}_2^{1/2} \big ) \norm{g}_2 \norm{h}_2^{1/2} \big (\norm{h}_2^{1/2} + \norm{\nabla_h h}_2^{1/2} \big ),
\end{split}
\end{equation}
and
\begin{equation}\nonumber
\begin{split}
& \int_{\Gamma_G} \bigg ( \int_0^1 f(x,y,z) \diff z \bigg) \bigg ( \int_0^1 g(x,y,z) h(x,y,z) \diff z \bigg ) \diff x \diff y \\
& \lesssim \norm{f}_2 \norm{g}_2^{1/2} \big ( \norm{g}_2^{1/2} + \norm{\nabla_h g}_2^{1/2} \big ) \norm{h}_2^{1/2} \big (\norm{h}_2^{1/2} + \norm{\nabla_h h}_2^{1/2} \big )
\end{split}
\end{equation}
hold true for any $f,g,h$ set of functions such that the right-hand sides make sense and are finite.

\end{lemma}

%Note that this lemma simply states the inequality for the complete domain even in the version in \cite{li2019primitive}, not assuming an odd $u_3^\epsilon$ function or a $z$-symmetrical vertical structure of $\Omega$.

The main idea in order to prove the strong convergence of $u^\epsilon_3$ lies in establishing a bound of $O(\epsilon^\kappa)$ for some $\kappa>0$ for the difference \begin{equation}\label{velocitydifferencefunctions}(\mathbf{U}_h^\epsilon, U_3^\epsilon): = (\mathbf{u}_h^\epsilon - \mathbf{u}_h, u_3^\epsilon - u_3).\end{equation} This will be achieved by using the strong solution $(\mathbf{u}_h,u_3)$ as a test function in the weak formulation's integral identity $(\ref{anisotropicWF_u})$. Following this method we arrive to

\begin{prop} \label{propIdentity}

Let $(\mathbf{u}_h^\epsilon, u_3^\epsilon)$ be a Leray-Hopf weak solution of (\ref{ANSC1}) - (\ref{ANSC4}), let $(\mathbf{u}_h,u_3)$ be the unique global strong solution of (\ref{HNSC1}) - (\ref{HNSC4}), moreover let the initial data as given in (\ref{ANSC_init}) and (\ref{HNSC_init}), furthermore we assume the regularities defined in (\ref{DWMC_weak_initial_data_regularities}). Then we have the following integral identity:

\begin{equation} \label{integralIdentity_prop}
\begin{split}
& - \frac{\epsilon^2}{2} \norm{u_3(t)}_2^2 + \bigg ( (\mathbf{u}_h^\epsilon \cdot \mathbf{u}_h + \epsilon^2 u_3^\epsilon u_3) \diff x \diff y \diff z \bigg ) (t) \\
& + \int_{Q_t} (-\mathbf{u}_h^\epsilon \cdot \partial_t \mathbf{u}_h + \nabla \mathbf{u}_h^\epsilon : \nabla \mathbf{u}_h + \epsilon^2 \nabla u_3^\epsilon \cdot \nabla u_3) \diff x \diff y \diff z \diff s \\
& = \frac{\epsilon^2}{2} \norm{{u_3}_0}_2^2 + \norm{{\mathbf{u}_h}_0}_2^2 + \epsilon^2 \int_{Q_t} \bigg ( \int_0^z \partial_t \mathbf{u}_h \diff z' \bigg ) \cdot \nabla_h U_3^\epsilon \diff x \diff y \diff z \diff s \\
& - \int_{Q_t} [(\mathbf{u}^\epsilon \cdot \nabla) \mathbf{u}_h^\epsilon \cdot \mathbf{u}_h + \epsilon^2 \mathbf{u}^\epsilon \cdot \nabla u_3^\epsilon u_3] \diff x \diff y \diff z \diff s \\
& + \int_{Q_t} ( \gamma u_2^\epsilon u_1 - \epsilon \beta u_3^\epsilon u_1 + \alpha \epsilon u_3^\epsilon u_2 - \gamma u_1^\epsilon u_2 - \epsilon \alpha u^\epsilon_2 u_3 + \epsilon \beta u_1^\epsilon u_3)
\end{split}
\end{equation}
for any $t \in [0, \infty),$ where $Q_t = \Omega \times (0,t).$

\end{prop}

\begin{proof}

We recall the definition of weak solution for the scaled Navier-Stokes equations:

\begin{equation} \label{WFreminder}
\begin{split}
& \int_Q [ - (\mathbf{u}_h^\epsilon \cdot \partial_t \tilde{\mathbf{u}}_h + \epsilon^2 u_3^\epsilon \partial_t \tilde{u}_3) +( (\mathbf{u}^\epsilon \cdot \nabla) \mathbf{u}_h^\epsilon \cdot \tilde{\mathbf{u}}_h + \epsilon^2 \mathbf{u}^\epsilon \cdot \nabla u_3^\epsilon \tilde{u}_3 ) + \nabla \mathbf{u}_h^\epsilon : \nabla \tilde{\mathbf{u}}_h \\
& + \epsilon^2 \nabla u_3^\epsilon \cdot \nabla \tilde{u}_3 ] \diff x \diff y \diff z \diff t = \int_Q ({\mathbf{u}_h}_0 \cdot \tilde{\mathbf{u}}_h (\cdot , 0) + \epsilon^2 {u_3}_0 \tilde{u}_3 (\cdot , 0)) \diff x \diff y \diff z \\
& + \int_Q \gamma u_2^\epsilon \tilde{u}_1 - \epsilon \beta u_3^\epsilon \tilde{u}_1 + \alpha \epsilon u_3^\epsilon \tilde{u}_2 - \gamma u_1^\epsilon \tilde{u}_2 - \epsilon \alpha u_2^\epsilon \tilde{u}_3 + \epsilon \beta u_1^\epsilon \tilde{u}_3
\end{split}
\end{equation}
for any $\tilde{\mathbf{u}} = (\tilde{\mathbf{u}}_h, \tilde{u}_3) \in H^1 (0,T;\mathbf{V})$ with $\tilde{\mathbf{u}}(T) = 0.$

Now, let's define $\chi \in C_0^\infty ([0, \infty))$ with $0 \leq \chi \leq 1$ and $\chi (0) = 1$ and set $\tilde{\mathbf{u}} = (\mathbf{u}_h,u_3) \chi (t).$ By using the density argument it is valid to choose this $\tilde{\mathbf{u}}$ function as a testing function as long as all the integral terms remain well-defined in $(\ref{WFreminder}),$ which we verify in the following.

It is important to highlight at this point that since we are not restricted to zero-averaged functions and we are not on a periodic domain, we need to pay special attention when verifying boundedness results using the Poincar\'e inequality. Note that as the value of the velocity functions are defined to be zero on the boundary, any first-order estimate of the type $\norm{\mathbf{u}_h^\epsilon} \leq \norm{\nabla \mathbf{u}_h^\epsilon}$ is valid, using the zero-trace version of the Poincar\'e inequality for functions in $W^{1,2}_0.$ However, since it is only the value of the functions that vanishes on the boundary, but not their derivatives, higher order estimates such as passing from $\norm{\nabla \mathbf{u}_h^\epsilon}$ to $\norm{\Delta \mathbf{u}_h^\epsilon}$ in the process of constructing an upper bound can not be applied in general. Which is why, for example, the estimate of the term below is rather involved as follows:

\begin{equation}
\begin{split}
& \int_Q \abs{\mathbf{u}^\epsilon} \abs{\nabla u_3^\epsilon} \abs{u_3} \chi \diff x \diff y \diff z \diff t \\
& \leq \int_0 ^ T \int_{\Gamma_G} \bigg ( \int_0^1 \abs{\mathbf{u}^\epsilon} \abs{\nabla u_3^\epsilon} \diff z \int_0^1 \abs{\nabla_h \mathbf{u}_h} \diff z \bigg ) \diff x \diff y \diff t \\
& \lesssim \int_0^T \norm{\mathbf{u}^\epsilon}_2^{1/2} \norm{\nabla \mathbf{u}^\epsilon}_2^{1/2} \norm{\nabla u_3^\epsilon}_2 \norm{\nabla_h \mathbf{u}_h}_2^{1/2} (\norm{\nabla_h \mathbf{u}_h}_2^{1/2} + \norm{\Delta \mathbf{u}_h}_2^{1/2}) \diff t \\
& \lesssim \sup_{0 \leq t \leq T} \bigg( \norm{\mathbf{u}^\epsilon}_2^{1/2} \norm{\nabla \mathbf{u}_h}_2^{1/2} \bigg) \bigg( \int_0^T \norm{\nabla \mathbf{u}^\epsilon}^2_2 \diff t \bigg)^{1/4} \bigg ( \int_0^T \norm{\nabla u_3^\epsilon}_2^2 \diff t \bigg)^{1/2} \\
& \times \bigg ( \int_0^T \norm{\Delta \mathbf{u}_h}_2^2 \diff t \bigg )^{1/4} \\
& + \sup_{0 \leq t \leq T} \bigg( \norm{\mathbf{u}^\epsilon}_2^{1/2} \norm{\nabla \mathbf{u}_h}_2 \bigg) \bigg( \int_0^T \norm{\nabla \mathbf{u}^\epsilon}_2 \diff t \bigg)^{1/2} \bigg ( \int_0^T \norm{\nabla u_3^\epsilon}_2^2 \diff t \bigg)^{1/2}.\raisetag{60pt}
\end{split}
\end{equation}

The validity of the remaining integral terms are not affected by the boundary change and they are guaranteed by the regularities of the respective solution functions and using basic inequalities. For completeness we remark that the more delicate and problematic terms -- i.e. the terms where the testing function is defined via the less regular $u_3$ vertical velocity -- are still well-defined. This can be verified firstly by rewriting the $\int_Q u^\epsilon_3 \partial_t \tilde{u}_3$ term as

$$ \int_Q u^\epsilon_3 \partial_t (u_3 \chi) \diff x \diff y \diff z \diff t = \int_0^\infty \langle \partial_t (u_3 \chi), u^\epsilon_3 \rangle _ {H^{-1} \times H^1} \diff t.$$

The validity of the integral term can be shown using this modified form pointing out that $\partial_t \mathbf{u}_h$ is in $L^2_{loc} ([0,\infty), L^2(\Omega))$ because of Theorem $\ref{PE3DWellPosedness}$, and as a consequence $\partial_t u_3$ has an $H^{-1}$ space regularity (using the divergence-free condition).

On the other hand, the $ \int_Q \nabla u^\epsilon_3 : \nabla (u_3 \chi) $ term is well-defined since $u_3$ as a strong solution has $H^1$ space regularity, and the same holds for the weak solution $u^\epsilon_3.$

Now we are justified to take $\tilde{\mathbf{u}}$ as a test function, and thus, rewriting some of the integral terms in $(\ref{WFreminder})$ in a more convenient form, we have

\begin{equation} \label{rewrittenform}
\begin{split}
& \int_Q (- \mathbf{u}_h^\epsilon \cdot \partial_t \mathbf{u}_h + \nabla \mathbf{u}_h^\epsilon : \nabla \mathbf{u}_h + \epsilon^2 \nabla u_3^\epsilon \cdot \nabla u_3) \chi \diff x \diff y \diff z \diff t \\
& - \epsilon^2 \int_0^\infty \langle \partial_t u_3 , u_3^\epsilon \rangle _{H^{-1} \times H^1} \chi \diff t - \int_Q (\mathbf{u}_h^\epsilon \cdot \mathbf{u}_h + \epsilon^2 u_3^\epsilon u_3) \chi' \diff x \diff y \diff z \diff t \\
& = - \int_Q [(\mathbf{u}^\epsilon \cdot \nabla) \mathbf{u}_h^\epsilon \cdot \mathbf{u}_h + \epsilon^2 \mathbf{u}^\epsilon \cdot \nabla u_3^\epsilon u_3] \chi \diff x \diff y \diff z \diff t + \norm{{\mathbf{u}_h}_0}_2^2 + \epsilon^2 \norm{{u_3}_0}_2^2 \\
& + \int_Q ( \gamma u_2^\epsilon u_1 - \epsilon \beta u_3^\epsilon u_1 + \alpha \epsilon u_3^\epsilon u_2 - \gamma u_1^\epsilon u_2 - \epsilon \alpha u^\epsilon_2 u_3 + \epsilon \beta u_1^\epsilon u_3) \chi,\raisetag{80pt}
\end{split}
\end{equation}
for any $\chi \in C^\infty_0 ([0,\infty)),$ with $0 \leq \chi \leq 1$ and $\chi (0) = 1.$

Since the estimate $(\ref{integralIdentity_prop})$ is localised in time, we take an arbitrary $t_0 \in (0, \infty)$ and a sufficiently small positive $\delta \in (0, t_0).$ Choose a cut-off function with the properties $\chi_\delta \in C^\infty_0 [0,t_0),$ such that $\chi_\delta \equiv 1$ on $[0, t_0 -\delta], 0 \leq \chi_\delta \leq 1$ on $[t_0 - \delta, t_0 )$ and $\abs{\chi_\delta '} \leq \frac{2}{\delta}$ on $[0, t_0).$ Now, taking $\chi = \chi_\delta$ in $(\ref{rewrittenform})$ and letting $\delta$ go to $0,$ one can arrive to the desired equality after performing the calculations for passing to the limit. It is easy to see that because of the velocities vanishing on the boundaries, the technical details remain unaffected, even though we are not on a periodic domain.

In more detail, the term $ \langle \partial_t u_3, u_3^\epsilon \rangle _{H^{-1} \times H^1}$ now takes the form 

\begin{equation}
\begin{split}
& \langle \partial_t u_3, u_3^\epsilon \rangle _{H^{-1} \times H^1} = - \bigg \langle \nabla_h \cdot \bigg ( \int_0^z \partial_t \mathbf{u}_h \diff z' \bigg ) , u_3^\epsilon \bigg \rangle \\
& = - \int_{\partial \Omega} \bigg ( \int_0^z \partial_t \mathbf{u}_h \diff z' \bigg ) u_3^\epsilon \vec{\boldsymbol{n}}_h + \int_\Omega \bigg ( \int_0^z \partial_t \mathbf{u}_h \diff z' \bigg) \cdot \nabla_h u_3^\epsilon \diff x \diff y \diff z,
\end{split}
\end{equation}
where $\vec{\boldsymbol{n}}_h$ is the horizontal component of the outward facing normal vector on the $\partial \Omega$ boundary.

By using that $u_3^\epsilon = 0$ on $\partial \Omega$ we arrive easily to the more simple

$$ \langle \partial_t u_3, u_3^\epsilon \rangle _{H^{-1} \times H^1} = \int_\Omega \bigg ( \int_0^z \partial_t \mathbf{u}_h \diff z' \bigg) \cdot \nabla_h u_3^\epsilon \diff x \diff y \diff z,$$
which is the form one would have on a periodic domain. Similarly, 

$$ \langle \partial_t u_3, u_3^\epsilon - u_3 \rangle _{H^{-1} \times H^1} = \int_\Omega \bigg ( \int_0^z \partial_t \mathbf{u}_h \diff z' \bigg) \cdot \nabla_h (u_3^\epsilon - u_3),$$ as the boundary term $$\int_{\partial \Omega} \bigg ( \int_0^z \partial_t \mathbf{u}_h \diff z' \bigg ) (u_3^\epsilon - u_3) \vec{\boldsymbol{n}}_h$$ becomes zero, since we have $u_3 = 0, u_3^\epsilon = 0$ on $\Gamma.$ 

After some technical details that we omit here we eventually obtain

\begin{equation}
\begin{split}
& - \frac{\epsilon^2}{2} \norm{u_3(t_0)}_2^2 + \bigg ( (\mathbf{u}_h^\epsilon \cdot \mathbf{u}_h + \epsilon^2 u_3^\epsilon u_3) \diff x \diff y \diff z \bigg ) (t_0) \\
& + \int_{Q_{t_0}} (-\mathbf{u}_h^\epsilon \cdot \partial_t \mathbf{u}_h + \nabla \mathbf{u}_h^\epsilon : \nabla \mathbf{u}_h + \epsilon^2 \nabla u_3^\epsilon \cdot \nabla u_3) \diff x \diff y \diff z \diff t \\
& = \frac{\epsilon^2}{2} \norm{{u_3}_0}_2^2 + \norm{{\mathbf{u}_h}_0}_2^2 + \epsilon^2 \int_{Q_{t_0}} \bigg ( \int_0^z \partial_t \mathbf{u}_h \diff z' \bigg ) \cdot \nabla_h U_3^\epsilon \diff x \diff y \diff z \diff t \\
& - \int_{Q_{t_0}} [(\mathbf{u}^\epsilon \cdot \nabla) \mathbf{u}_h^\epsilon \cdot \mathbf{u}_h + \epsilon^2 \mathbf{u}^\epsilon \cdot \nabla u_3^\epsilon u_3] \diff x \diff y \diff z \diff t \\
& + \int_{Q_{t_0}} ( \gamma u_2^\epsilon u_1 - \epsilon \beta u_3^\epsilon u_1 + \alpha \epsilon u_3^\epsilon u_2 - \gamma u_1^\epsilon u_2 - \epsilon \alpha u_2^\epsilon u_3 + \epsilon \beta u_1^\epsilon u_3)
\end{split}
\end{equation}
for any $t_0 \in [0, \infty),$ which verifies the statement of the proposition.

\end{proof}

With this we are ready to estimate the difference between the solutions of the (\ref{ANSC1}) - (\ref{ANSC4}) scaled and the (\ref{HNSC1}) - (\ref{HNSC4}) hydrostatic equations, which is described by the following proposition.

\begin{prop} \label{convrateweakframework}

Let $(\mathbf{u}_h^\epsilon, u_3^\epsilon)$ be a Leray-Hopf weak solution of (\ref{ANSC1}) - (\ref{ANSC4}), let $(\mathbf{u}_h,u_3)$ be the unique global strong solution of (\ref{HNSC1}) - (\ref{HNSC4}), moreover let the initial data as given in (\ref{ANSC_init}) and (\ref{HNSC_init}), furthermore we assume the regularities defined in (\ref{DWMC_weak_initial_data_regularities}). Then the following estimate holds for the difference function $(\mathbf{U}_h^\epsilon, U_3^\epsilon):$

\begin{equation}\label{horizontal_velocities_convergence}
\begin{split}
& \sup_{0 \leq t < \infty} ( \norm{\mathbf{U}_h^\epsilon}_2^2 + \epsilon^2 \norm{U_3^\epsilon}_2^2)(t) + \int_0 ^ \infty ( \norm{\nabla \mathbf{U}_h^\epsilon}_2^2 + \epsilon^2 \norm{\nabla U_3^\epsilon}_2^2)\diff s \\
& \leq \kappa (\norm{{\mathbf{u}_h}_0}_{H^1}, \Gamma_G) \epsilon^2 (\norm{{\mathbf{u}_h}_0}_2^2 + \epsilon^2 \norm{{u_3}_0}_2^2 + 1).
\end{split}
\end{equation}

\end{prop}

\begin{proof}

Multiplying the horizontal part of the hydrostatic equations, i.e. equations ($\ref{HNSC1}$) - ($\ref{HNSC2}$) by $\mathbf{u}_h^\epsilon$ and integrating the result over $Q_{t_0},$ by utilising the divergence-free condition and the zero velocity boundary conditions, we have

\begin{equation} \label{identity_obtained_by_v_eps}
\begin{split}
& \int_{Q_{t_0}} (\partial_t \mathbf{u}_h \cdot \mathbf{u}_h^\epsilon + \nabla \mathbf{u}_h : \nabla \mathbf{u}_h^\epsilon) \diff x \diff y \diff z \diff t \\
& = - \int_{Q_{t_0}} (\mathbf{u} \cdot \nabla) \mathbf{u}_h \cdot \mathbf{u}_h^\epsilon + \gamma u_2 u_1^\epsilon - \gamma u_1 u_2^\epsilon \diff x \diff y \diff z \diff t.
\end{split}
\end{equation}

On the other hand, multiplying the same ($\ref{HNSC1}$) - ($\ref{HNSC2}$) equations by $\mathbf{u}_h$ instead, and integrating again on $Q_{t_0},$ we arrive to

\begin{equation} \label{identity_obtained_by_v}
\frac{1}{2} \norm{\mathbf{u}_h(t_0)}_2^2 + \int_0^{t_0} \norm{\nabla \mathbf{u}_h}_2^2 \diff t = \frac{1}{2} \norm{{\mathbf{u}_h}_0}_2^2,
\end{equation}
for any $t_0 \in [0, \infty).$ At the same time, by the definition of Leray-Hopf weak solutions, we also know that

\begin{equation} \label{bound_by_definition}
\begin{split}
& \frac{1}{2}( \norm{\mathbf{u}_h^\epsilon (t_0)}_2^2 + \epsilon^2 \norm{u_3^\epsilon(t_0)}_2^2)+ \int_0^{t_0} (\norm{\nabla \mathbf{u}_h^\epsilon}_2^2 + \epsilon^2 \norm{\nabla u_3^\epsilon}_2^2) \diff s \\
& \leq \frac{1}{2} (\norm{{\mathbf{u}_h}_0}_2^2 + \epsilon^2 \norm{{u_3}_0}_2^2),
\end{split}
\end{equation}
for almost every $t_0 \in [0, \infty).$

Now, summing ($\ref{bound_by_definition}$) and ($\ref{identity_obtained_by_v}$) and from their result subtracting ($\ref{integralIdentity_prop}$) (where we take $t= t_0$) and ($\ref{identity_obtained_by_v_eps}$) we obtain

\begin{equation} \label{introduce_I_terms}
\begin{split}
& \frac{1}{2} (\norm{\mathbf{U}_h^\epsilon}_2^2 + \epsilon^2 \norm{U_3^\epsilon}_2^2) (t_0) + \int_0^{t_0} (\norm{\nabla \mathbf{U}_h^\epsilon}_2^2 + \epsilon^2 \norm{\nabla U_3^\epsilon}_2^2) \diff t \\
& \leq - \epsilon^2 \int_{Q_{t_0}} \bigg [ \bigg ( \int_0^z \partial_t \mathbf{u}_h \diff z' \bigg ) \cdot \nabla_h U_3^\epsilon + \nabla u_3 \cdot \nabla U_3^\epsilon \bigg ] \diff x \diff y \diff z \diff t \\
& + \int_{Q_{t_0}} [(\mathbf{u}^\epsilon \cdot \nabla) \mathbf{u}_h^\epsilon \cdot \mathbf{u}_h + (\mathbf{u} \cdot \nabla )\mathbf{u}_h \cdot \mathbf{u}_h^\epsilon] \diff x \diff y \diff z \diff t \\
& + \epsilon^2 \int_{Q_{t_0}} \mathbf{u}^\epsilon \cdot \nabla u_3^\epsilon u_3 \diff x \diff y \diff z \diff t \\
& + \int_{Q_{t_0}} \epsilon \beta u_3^\epsilon u_1 - \epsilon \alpha u_3^\epsilon u_2 + \epsilon \alpha u_2^\epsilon u_3 - \epsilon \beta u_1^\epsilon u_3 \diff x \diff y \diff z \diff t := I_1 + I_2 + I_3 + I_4, \raisetag{80pt}
\end{split}
\end{equation}
for almost every $t_0 \in [0, \infty).$

The next step of the proof is to provide a bound for each of the $I_1, I_2, I_3, I_4$ terms above, in which the regularity of the strong solutions of the primitive equations (see Corollary $\ref{StrongSolPE_boundedness}$) plays a crucial role.

\textbf{Bound for $I_1.$} It can be verified easily that for $I_1$ the following bound holds:

$$ I_1 \lesssim \zeta \epsilon^2 \norm{\nabla U_3^\epsilon}^2_{L^2(Q_{t_0})} + \epsilon^2 \kappa (\norm{{\mathbf{u}_h}_0}_{H^1}, \Gamma_G). $$

\textbf{Bound for $I_2.$} Calculating the bound for the $I_2$ term we arrive to a different expression compared the one obtained in \cite{li2019primitive} --- again, this is a result of the modified structure of the domain, and the appearance of new terms that had previously been implicitly merged into other higher order terms. 
%Some of the computational details that are independent and unaffected by the boundary change will be omitted since they are analogous to \cite{li2019primitive}. It is important to note however that some of the expressions remain identical in both versions, but they are correct for different reasons, and thus, most of these steps will nevertheless be written out in detail.

Firstly, applying integration by parts, the zero boundary condition for the velocity functions, and the divergence-free condition, $I_2$ becomes

\begin{equation}
\begin{split}
& I_2 = \int_{Q_{t_0}} [(\mathbf{u}^\epsilon \cdot \nabla) \mathbf{u}_h^\epsilon \cdot \mathbf{u}_h + (\mathbf{u} \cdot \nabla )\mathbf{u}_h \cdot \mathbf{u}_h^\epsilon] \diff x \diff y \diff z \diff t \\
& = \int_{Q_{t_0}} [(\mathbf{u}^\epsilon \cdot \nabla) \mathbf{u}_h^\epsilon \cdot \mathbf{u}_h - (\mathbf{u} \cdot \nabla )\mathbf{u}^\epsilon_h \cdot \mathbf{u}_h] \diff x \diff y \diff z \diff t \\
& = \int_{Q_{t_0}} [(\mathbf{u}^\epsilon - \mathbf{u}) \cdot \nabla ] \mathbf{u}_h^\epsilon \cdot \mathbf{u}_h \diff x \diff y \diff z \diff t \\
& = \int_{Q_{t_0}} [(\mathbf{u}^\epsilon - \mathbf{u}) \cdot \nabla ] (\mathbf{u}_h^\epsilon - \mathbf{u}_h) \cdot \mathbf{u}_h \diff x \diff y \diff z \diff t \\
& = \int_{Q_{t_0}} [(\mathbf{U}_h^\epsilon, U^\epsilon_3) \cdot \nabla ] \mathbf{U}_h^\epsilon \cdot \mathbf{u}_h \diff x \diff y \diff z \diff t := I_2' + I_2''.
\end{split}
\end{equation}

In order to estimate $I_2',$ we will use the zero velocity boundary conditions which enable us to apply $\norm{\mathbf{u}_h}_{L^6} \leq C \norm{\mathbf{u}_h}_{H^1} \leq C \norm{\nabla \mathbf{u}_h}_2.$ Combining this with the Sobolev, Young and H\"older inequalities, we arrive to

\begin{equation} \label{I_2_prime}
\begin{split}
& I_2'= \int_{Q_{t_0}} (\mathbf{U}^\epsilon_h \cdot \nabla_h) \mathbf{U}^\epsilon_h \cdot \mathbf{u}_h \diff x \diff y \diff z \diff t \\
& \leq \int_0^{t_0} \| \mathbf{U}^\epsilon_h \| _3 \| \nabla \mathbf{U}^\epsilon_h \| _2 \norm{\mathbf{u}_h}_6 \diff t \lesssim \int_0^{t_0} \| \mathbf{U}^\epsilon_h \| _2^{1/2} \| \nabla \mathbf{U}^\epsilon_h \|_2^{3/2} \norm{\nabla \mathbf{u}_h}_2 \diff t \\
& \lesssim \zeta \| \nabla \mathbf{U}^\epsilon_h \| ^2_{L^2(Q_{t_0})} + \kappa(\zeta) \int_0^{t_0} \norm{\nabla \mathbf{u}_h}_2^4 \| \mathbf{U}^\epsilon_h \| _2^2 \diff t.
\end{split}
\end{equation}
On the other hand, estimating $I_2'',$ we obtain
\begin{equation}
\begin{split}
& I_2 '' = \int_{Q_{t_0}} U^\epsilon_3 \partial_z \mathbf{U}_h^\epsilon \cdot \mathbf{u}_h \diff x \diff y \diff z \diff t \\
& = \int_{Q_{t_0}} [\nabla_h \cdot \mathbf{U}_h^\epsilon \mathbf{U}_h^\epsilon \mathbf{u}_h - U^\epsilon_3 \mathbf{U}_h^\epsilon \cdot \partial_z \mathbf{u}_h ] \diff x \diff y \diff z \diff t := I_{21}'' + I_{22}'',
\end{split}
\end{equation}
where $I_{21}''$ can be bounded by the same term as $I_2'$ in $(\ref{I_2_prime}),$ while for $I_{22}''$ we will follow the steps below in order to construct an estimate.

\begin{equation} \nonumber
\begin{split}
& I_{22}'' = - \int_{Q_{t_0}} U^\epsilon_3 \mathbf{U}_h^\epsilon \cdot \partial_z \mathbf{u}_h \diff x \diff y \diff z \diff t \\
& \leq \int_0^{t_0} \int_\Omega \bigg ( \int_0^z \nabla_h \cdot \mathbf{U}_h^\epsilon \diff z' \bigg ) (\mathbf{U}_h^\epsilon \cdot \partial_z \mathbf{u}_h) \diff x \diff y \diff z \diff t \\
& \lesssim \int_0^{t_0} \int_{\Gamma_G} \bigg ( \int_0^1 \abs{\nabla_h \mathbf{U}_h^\epsilon} \diff z \bigg) \bigg ( \int_0^1 \abs{\mathbf{U}_h^\epsilon} \abs{\partial_z \mathbf{u}_h} \diff z \bigg ) \diff x \diff y \diff t \\
& \lesssim \int_0^{t_0} \| \nabla \mathbf{U}_h^\epsilon \| _2^{3/2} \| \mathbf{U}_h^\epsilon \| _2^{1/2} \norm{\nabla \mathbf{u}_h}_2^{1/2} (\norm{\nabla \mathbf{u}_h}_2^{1/2} + \norm{\Delta \mathbf{u}_h}_2^{1/2}) \diff t \\
& \lesssim \zeta \| \nabla \mathbf{U}_h^\epsilon \| ^2_2 + \kappa(\zeta) \int_0^{t_0} \norm{\nabla \mathbf{u}_h}_2^2 (\norm{\nabla \mathbf{u}_h}_2^{1/2} + \norm{\Delta \mathbf{u}_h}_2^{1/2})^4 \| \mathbf{U}_h^\epsilon \|_2^2 \diff t \\
& \lesssim \zeta \| \nabla \mathbf{U}_h^\epsilon \| ^2_2 + \kappa(\zeta) \int_0^{t_0} \norm{\nabla \mathbf{u}_h}_2^4 \| \mathbf{U}_h^\epsilon \|_2^2 \diff t + \kappa(\zeta) \int_0^{t_0} \norm{\nabla \mathbf{u}_h}_2^2 \norm{\Delta \mathbf{u}_h}_2^2 \| \mathbf{U}_h^\epsilon \| _2^2 \diff t,
\end{split}
\end{equation}
where we utilised the divergence-free condition, the Young inequality, the $\| \mathbf{U}_h^\epsilon \|_2 \leq \| \nabla \mathbf{U}_h^\epsilon \|_2$ first-order bound thanks to the zero boundary conditions, and we applied Lemma $\ref{LadyzhenskayatypeIEQ}$ using

$$f = \nabla_h \mathbf{U}_h^\epsilon, g = \mathbf{U}_h^\epsilon, h = \partial_z \mathbf{u}_h.$$

We highlight the presence of the new, additional term $\int_0^{t_0} \norm{\nabla \mathbf{u}_h}_2^4 \| \mathbf{U}_h^\epsilon \| _2^2 \diff t$ in the bound compared to the estimate obtained in \cite{li2019primitive}. The difference is manifested at the point when we apply Lemma $\ref{LadyzhenskayatypeIEQ}$, since we can not apply the Poincar\'e inequality for $\nabla \mathbf{u}_h,$ and as a consequence, instead of $\norm{h}_2^{1/2} \norm{\nabla h}_2^{1/2}$ we will have the complete $\norm{h}_2^{1/2} (\norm{h}_2^{1/2} + \norm{\nabla h}_2^{1/2})$ expression, which leads to the appearance of $\kappa(\zeta) \int_0^{t_0} \norm{\nabla \mathbf{u}_h}_2^4 \| \mathbf{U}_h^\epsilon \|_2^2 \diff t.$

Combining the bounds for $I_2', I_{21}''$ and $I_{22}'',$ we can estimate the $I_2$ term by
$$ I_2 \lesssim \zeta \| \nabla \mathbf{U}_h^\epsilon \| ^2_2 + \kappa(\zeta) \int_0^{t_0} \norm{\nabla \mathbf{u}_h}_2^4 \| \mathbf{U}_h^\epsilon \| _2^2 \diff t + \kappa(\zeta) \int_0^{t_0} \norm{\nabla \mathbf{u}_h}_2^2 \norm{\Delta \mathbf{u}_h}_2^2 \| \mathbf{U}_h^\epsilon \| _2^2 \diff t. $$

\textbf{Bound for $I_3.$} Exploiting the incompressibility condition and the zero velocity boundary conditions, we have

\begin{equation}
\begin{split}
& I_3 = \epsilon^2 \int_{Q_{t_0}} \mathbf{u}^\epsilon \cdot \nabla u_3^\epsilon u_3 \diff x \diff y \diff z \diff t \\
& \leq \epsilon^2 \int_0^{t_0} \int_{\Gamma_G} \bigg ( \int_0^1 \abs{\mathbf{u}_h^\epsilon} \abs{\nabla_h U_3^\epsilon} + \abs{u_3^\epsilon} \abs{\nabla_h \mathbf{U}_h^\epsilon}\diff z \bigg) \bigg ( \int_0^1 \abs{\nabla_h \mathbf{u}_h} \diff z \bigg ) \diff x \diff y \diff t \\
& \lesssim \epsilon^2 \int_0^{t_0} ( \| \mathbf{u}_h^\epsilon \| _2^{1/2} \| \nabla\mathbf{u}_h^\epsilon \| _2^{1/2} \norm{\nabla U_3^\epsilon}_2 \\
& + \norm{u_3^\epsilon}_2^{1/2} \norm{\nabla u_3^\epsilon}_2^{1/2} \| \nabla_h \mathbf{U}_h^\epsilon \|_2) \norm{\nabla \mathbf{u}_h }_2^{1/2} (\norm{\nabla \mathbf{u}_h}_2^{1/2} + \norm{\Delta \mathbf{u}_h}_2^{1/2}) \diff t \\
& \lesssim \kappa(\zeta) \epsilon^2 \int_0 ^ {t_0} \bigg ( \| \mathbf{u}_h^\epsilon \| _2^2 \| \nabla \mathbf{u}_h^\epsilon \|_2^2 + \norm{\nabla \mathbf{u}_h}_2^2 (\norm{\nabla \mathbf{u}_h}_2^{1/2} + \norm{\Delta \mathbf{u}_h}_2^{1/2})^4 \\
& + \epsilon^2 \norm{u_3^\epsilon}_2^2 \norm{\nabla u_3^\epsilon}_2^2 \bigg ) \diff t + \zeta \big ( \| \nabla \mathbf{U}_h^\epsilon \|_{L^2(Q_{t_0})}^2 + \epsilon^2 \norm{\nabla U_3^\epsilon}_{L^2(Q_{t_0})}^2 \big),
\end{split}
\end{equation}
where we used the Young inequality and Lemma \ref{LadyzhenskayatypeIEQ}, keeping in mind that passing from $\norm{\nabla \mathbf{u}_h}$ to $\norm{\Delta \mathbf{u}_h}$ in the estimating process is not legitimate. Note at the same time that in this case the additional $\norm{\nabla \mathbf{u}_h}_2^2 \cdot \norm{\nabla \mathbf{u}_h}_2^2$ term will not have a permanent role in the bound, as in the following step it will be estimated by the same $\kappa (\norm{{\mathbf{u}_h}_0}_{H^1}, \Gamma_G)$ term that is used to estimate the Laplacian. Specifically, using Corollary $\ref{StrongSolPE_boundedness}$ we arrive back to the same conclusion ($\ref{I3Bound}$) that was calculated for the boundary-free scenario:

\begin{equation} \label{I3Bound}
\begin{split}
& I_3 \leq \zeta \big ( \| \nabla \mathbf{U}_h^\epsilon \| _{L^2(Q_{t_0})}^2 + \epsilon^2 \norm{\nabla U_3^\epsilon}_{L^2(Q_{t_0})}^2 \big )\\
& + \kappa \epsilon^2 [(\norm{{\mathbf{u}_h}_0}_2^2 + \epsilon^2 \norm{{u_3}_0}_2^2)^2 + \kappa( \| {\mathbf{u}_h}_0 \| _{H^1}, \Gamma_G)],
\end{split}
\end{equation}
where we also used $(\ref{bound_by_definition})$.

\textbf{Bound for $I_4.$} We add $ - \epsilon \beta u_3 u_1 + \epsilon \beta u_3 u_1 - \epsilon \alpha u_3 u_2 + \epsilon \alpha u_3 u_2 = 0 $ to $I_4$, thus instead of estimating the original version of the term, now we need to bound the equivalent term

$$ I_4 = \int_{Q_{t_0}} ( \epsilon \alpha U_2^\epsilon u_3 - \epsilon \beta U_1^\epsilon u_3 + \epsilon \beta U_3^\epsilon u_1 - \epsilon \alpha U_3^\epsilon u_2 ) \diff x \diff y \diff z \diff t = I_{41} + I_{42} + I_{43} + I_{44}.$$

%\begin{itemize}

\item Firstly, we have

\begin{equation}
\begin{split}
I_{41} = & \int_{Q_{t_0}} \epsilon \alpha U_2^\epsilon u_3 \diff x \diff y \diff z \diff t \lesssim \int_0^{t_0} \norm{U_2^\epsilon} \norm{\epsilon u_3} \diff t \\
& \lesssim \int_0^{t_0} \zeta \norm{U_2^\epsilon}^2_2 \diff t + \int_0^{t_0} \frac{1}{\zeta} \epsilon^2 \norm{u_3}^2_2 \diff t \\
& \lesssim \int_0^{t_0} \zeta \norm{\nabla U_2^\epsilon}^2_2 \diff t + \int_0^{t_0} \frac{1}{\zeta} \epsilon^2 \norm{u_3}^2_2 \diff t.
\end{split}
\end{equation}

Here the $\int_0^{t_0} \zeta \norm{\nabla U_2^\epsilon}^2_2$ term can be merged into the left-hand side of ($\ref{introduce_I_terms}$), while

$$\int_0^{t_0} \frac{1}{\zeta} \epsilon^2 \norm{u_3}^2_2 \leq \frac{1}{\zeta} \epsilon^2 \norm{u_3}^2_{L^2(Q_{t_0})} \leq \frac{1}{\zeta} \epsilon^2 \kappa (\norm{{\mathbf{u}_h}_0}_{H^1}, \Gamma_G),$$
where we used the Corollary $\ref{StrongSolPE_boundedness}$, and the fact that $\norm{u_3} \leq \norm{\nabla \mathbf{u}_h}.$
Here $\zeta$ is the constant of the Young inequality, and it only has to be small enough to keep the coefficient of $\norm{\nabla \mathbf{U}_h^\epsilon}$ positive.

 The estimate of the $ I_{42} = \epsilon \beta U_1^\epsilon u_3$ term is analogous.

 For the term $ I_{43} = \epsilon \beta U_3^\epsilon u_1$ we can use the idea that $\| U_3^\epsilon \| \leq \| \nabla \mathbf{U}_h^\epsilon \|.$ As we have $\partial_3 u_3^\epsilon = \nabla_h \mathbf{u}_h^\epsilon,$ and $\partial_3 u_3 = \nabla_h \mathbf{u}_h,$ we naturally have $\partial_3 U_3^\epsilon = \nabla_h \mathbf{U}_h^\epsilon,$ and by the vertical Poincar\'e inequality we get
 $$\norm{U_3^\epsilon} \leq \norm{\partial_3 U_3^\epsilon} = \norm{\nabla \mathbf{U}_h^\epsilon}.$$ 
% This way we are able to bound the $U_3^\epsilon$ term without keeping the $\epsilon^2$ multiplier attached to it when we apply the Young inequality. We have
Therefore we have
$$ \hspace{-1em}\int_{Q_{t_0}} \epsilon \beta U_3^\epsilon u_1 \diff x \diff y \diff z \diff t \lesssim \int_0^{t_0} \epsilon \norm{U_3^\epsilon}_2 \norm{u_1}_2 \diff t \lesssim \int_0^{t_0} \epsilon \| \nabla \mathbf{U}_h^\epsilon \|_2 \norm{u_1}_2 \diff t.$$
and finally, after applying the Young inequality we arrive to
\begin{equation}
\begin{split}
I_{43} = & \int_{Q_{t_0}} \epsilon \beta U_3^\epsilon u_1 \diff x \diff y \diff z \diff t \lesssim \zeta \int_0^{t_0} \norm{\nabla \mathbf{U}_h^\epsilon}_2^2 \diff s + \frac{1}{\zeta} \epsilon^2 \int_0^{t_0} \norm{u_1}_2^2 \diff s \\
& \lesssim \zeta \int_0^{t_0} \norm{\nabla \mathbf{U}_h^\epsilon}_2^2 \diff t + \frac{1}{\zeta} \epsilon^2 \kappa (\norm{{\mathbf{u}_h}_0}_{H^1}, \Gamma_G),
\end{split}
\end{equation}
%
%where the first term can be merged into the left-hand side of ($\ref{introduce_I_terms}$).

The estimate of $ I_{44} = \alpha \epsilon U_3^\epsilon u_2$ is structurally analogous.

%\end{itemize}

Using the bounds we now have for the terms $I_1, I_2, I_3$ and $I_4$, ($\ref{introduce_I_terms}$) takes the form
\begin{equation}
\begin{split} \nonumber
\Phi (t):= & (\norm{\mathbf{U}_h^\epsilon}_2^2 + \epsilon^2 \norm{U_3^\epsilon}_2^2)(t) + \int_0^t (\norm{\nabla \mathbf{U}_h^\epsilon}_2^2 + \epsilon^2 \norm{\nabla U_3^\epsilon}_2^2) \diff s \\
& \lesssim \epsilon^2 [(\norm{{\mathbf{u}_h}_0}_2^2 + \epsilon^2 \norm{{u_3}_0}_2^2)^2 + \kappa (\norm{{\mathbf{u}_h}_0}_{H^1}, \Gamma_G)] \\
& + \int_0^{t} \norm{\nabla \mathbf{u}_h}_2^4 \norm{\mathbf{U}_h^\epsilon}_2^2 \diff s + \int_0^{t} \norm{\nabla \mathbf{u}_h}_2^2 \norm{\Delta \mathbf{u}_h}_2^2 \norm{\mathbf{U}_h^\epsilon}_2^2 \diff s =: \Psi(t),
\end{split}
\end{equation}
for almost any $t_0 \in [0, \infty).$ With this we have
\begin{equation}
\begin{split} \nonumber
\Psi'(t) & = \norm{\nabla \mathbf{u}_h}_2^4 \norm{\mathbf{U}_h^\epsilon}_2^2 + \norm{\nabla \mathbf{u}_h}_2^2 \norm{\Delta \mathbf{u}_h}_2^2 \norm{\mathbf{U}_h^\epsilon}_2^2 \\
& \lesssim (\norm{\nabla \mathbf{u}_h}_2^4 + \norm{\nabla \mathbf{u}_h}_2^2 \norm{\Delta \mathbf{u}_h}_2^2) \Phi (t) \lesssim (\norm{\nabla \mathbf{u}_h}_2^4 + \norm{\nabla \mathbf{u}_h}_2^2 \norm{\Delta \mathbf{u}_h}_2^2)\Psi(t),
\end{split}
\end{equation}
from which, applying the Gr\"onwall inequality and using Corollary $\ref{StrongSolPE_boundedness}$ we deduce
\begin{equation}
\begin{split} \nonumber
& \Phi (t) \lesssim \Psi (t) \leq e^{\kappa \int_0 ^ t (\norm{\nabla \mathbf{u}_h}_2^4 + \norm{\nabla \mathbf{u}_h}_2^2 \norm{\Delta \mathbf{u}_h}_2^2) \diff s} \Psi (0) \leq e^{\kappa(\norm{{\mathbf{u}_h}_0}_{H^1}, \Gamma_G)} \Psi(0) \\
& \leq \epsilon^2 \kappa (\norm{{\mathbf{u}_h}_0}_{H^1}, \Gamma_G) (\norm{{\mathbf{u}_h}_0}_2^2 + \epsilon^2 \norm{{u_3}_0}_2^2 + 1)^2,
\end{split}
\end{equation}
which concludes the proof.

\end{proof}
Thus, we now have
$$ (\norm{\mathbf{U}_h^\epsilon}_2^2 + \epsilon^2 \norm{U_3^\epsilon}_2^2)(t) + \int_0^t (\norm{\nabla \mathbf{U}_h^\epsilon}_2^2 + \epsilon^2 \norm{\nabla U_3^\epsilon}_2^2) \diff s \leq \epsilon^2 \kappa,$$
where the $\kappa$ constant only depends on the domain structure and the initial data. With this we arrive to the strong convergence result $$\nabla \mathbf{U}_h^\epsilon \to 0 \text{ in } L^2_t L^2_x,$$ and as a direct consequence, we have 
$$\nabla \mathbf{u}_h^\epsilon \to \nabla \mathbf{u}_h \text{ in } L^2_t L^2_x$$ 
as well. Using the relation between the horizontal and vertical velocity functions through the divergence-free condition, we immediately obtain $$\partial_3 u_3^\epsilon \to \partial_3 u_3 \text{ in } L^2_t L^2_x,$$ and using the vertical Poincar\'e inequality eventually we deduce that
\begin{equation}
\label{strongconvresultu3}u_3^\epsilon \to u_3 \text{ strongly in } L^2_t L^2_x.
\end{equation}

This has been the last missing part of the proof of the main theorem, since combining the weak convergence result ($\ref{WeakConvCEps1}$) for $c^\epsilon$ and the strong convergence ($\ref{strongconvresultu3}$) of $u_3^\epsilon,$ we are finally able to state that
\begin{equation}c^\epsilon u_3^\epsilon \rightharpoonup c u_3 \text{ weakly in } L_t^2 L_x^1,\end{equation}
and to pass into the limit in the remaning nonlinear terms and we  conclude the proof of Theorem $\ref{maintheoremWEAK}$.

%%% STRONG SOLUTIONS %%%
\section{Strong solutions framework} \label{sectionSTRONGsol}

In this section we verify that the previously shown connection between the (\ref{ANSC1}) - (\ref{ANSC5}) anisotropic and the (\ref{HNSC1}) - (\ref{HNSC5}) hydrostatic equations has an analogous version that holds in the scenario of strong solutions as well. In order to describe our result we firstly need to give exact definition to the domain and set the boundary conditions. Here we follow the main ideas of \cite{cao2014local} and create a periodic domain $\Omega$ in three steps as described below.

\begin{remark}
\normalfont In the weak solutions' framework we defined a local domain with classical closed boundaries. Preserving this non-periodic domain in the case of strong solutions would naturally require higher order terms to be zero on the boundaries to guarantee that the a priori bounds can still be performed.
%(otherwise the computations can not be brought to a closure, consider for example Lemma 2.2 in \cite{li2019primitive} in a non-periodic case). 
In this framework we essentially need a domain structure where the $H^2$ norm is equivalent to the $L^2$ norm of the Laplacian, hence, for simplicity we make the technically more clear choice of assuming periodic boundaries. 
%At the same time we point out that theoretically, the periodic boundary conditions could be substituted by boundary conditions of type $\nabla^k u = 0, \nabla^k c = 0, k \geq 1$ defined on a non-periodic, standard domain. By applying these alternative boundary conditions it would be possible to switch back to a classical domain and still benefit from the disappearing boundary integrals.
We highlight that aligning the present framework's initial regularities (see later in (\ref{stronginitdataregularities})) with a traditional closed domain would require further considerations (specifically, domains with corners are no longer compatible with the required smoothness of the initial data).
\end{remark}

The concept detailed here is a general approach, and we apply it both in the hydrostatic and anisotropic cases.

Let $\Omega_0 = \Omega_2 \times (-a, 0)$ with $\Omega_2 = (0,1) \times (0,1)$ over an inland domain. We define the boundary conditions

\begin{equation}
\begin{split}
& \mathbf{u}_h, u_3 \text{ and } c \text{ are periodic in } x \text{ and } y, \\
& (\partial_3 \mathbf{u}_h, u_3) | _{z = -a, 0} = 0, \\
& c |_{z = -a} = 1, c |_{z = 0} = 0.
\end{split}
\end{equation}

Defining a new concentration function $c^*$ by setting $c^* = c + \frac{z}{a},$ we obtain a function that is zero on both the upper and lower boundary of the domain. This enables us to consider the standard extensions (see also \cite{li2019primitive}) of the horizontal velocity, vertical velocity, pressure and concentration functions to the new $\Omega = \Omega_2 \times (-a, a)$ domain by requiring these functions to be even, odd, even and odd in $z$, respectively, and set the domain to be periodic in all three dimensions.

Applying these ideas to the original concentration dynamic equations (\ref{ANSC5}) and (\ref{HNSC5}), we obtain that
\begin{equation} \label{HNSC5_strong_cshift}
\begin{split}
\partial_t c + \mathbf{u}_h \cdot \nabla_h c + u_3 \bigg ( \partial_3 c - \frac{1}{a} \bigg ) = \Delta_d c + s \text{ in } \Omega \times (0,T)
\end{split}
\end{equation}
is the new form of the hydrostatic concentration equation, while the governing equation for the anisotropic case becomes

\begin{equation} \label{ANSC5_strong_cshift}
\partial_t c^\epsilon + \mathbf{u}_h^\epsilon \cdot \nabla_h c^\epsilon + u_3^\epsilon \bigg ( \partial_3 c^\epsilon - \frac{1}{a} \bigg )= \epsilon K_1 \partial_{11} ^ 2 c^\epsilon + \Delta_d c^\epsilon + s^\epsilon \text{ in } \Omega \times (0,T),
\end{equation}
where after using $c = c^* - \frac{z}{a}$ and $c^\epsilon = {c^\epsilon}^* - \frac{z}{a},$ we returned to the original $c$ and $c^\epsilon$ notations for simplicity.

In terms of boundary conditions now the (\ref{ANSC1}) - (\ref{ANSC4}), (\ref{ANSC5_strong_cshift}), (\ref{ANSC_init}) anisotropic system the and (\ref{HNSC1}) - (\ref{HNSC4}), (\ref{HNSC5_strong_cshift}), (\ref{HNSC_init}) hydrostatic system are combined with the
\begin{equation} \label{ANSC_bc-strong}
\begin{split}
& \mathbf{u}^\epsilon, p^\epsilon \text{ and } c^\epsilon \text{ are periodic in } x,y,z, \\
& \mathbf{u}_h^\epsilon \text{ and } p^\epsilon \text{ are even in $z$}, \text{ and } u_3^\epsilon \text{ and } c^\epsilon \text{ are odd in $z$}
\end{split}
\end{equation}
and
\begin{equation} \label{HNSC_bc-strong}
\begin{split}
& \mathbf{u}, p \text{ and } c \text{ are periodic in } x,y,z \\
& \mathbf{u}_h \text{ and } p \text{ are even in $z$}, \text{ and } u_3 \text{ and } c \text{ are odd in $z$}
\end{split}
\end{equation}
conditions, respectively.

It is easy to check that the restriction of a solution to (\ref{ANSC1}) - (\ref{ANSC4}), (\ref{ANSC5_strong_cshift}), (\ref{ANSC_init}), (\ref{ANSC_bc-strong}) from $\Omega$ to $\Omega_0$ is a solution to the original system, and the analogous statement holds for the hydrostatic equations, hence here we exclusively focus on the study of these new versions of the two main systems.

We highlight here that the steps detailed above constructing a fully periodic domain are indeed necessary since, generally speaking, it is physically meaningful to consider the Earth's atmosphere as a domain that is periodic in the horizontal dimensions, however it is highly unnatural to immediately assume vertical periodicity as well. This strategy of adjusting the boundary value of the concentration function and extending the functions according to certain symmetries shows that the setting of a fully periodic domain can actually be a well-grounded choice for the atmosphere, given that the required symmetries hold.

\begin{remark}
\normalfont Note that the $\Omega$ domain we construct this way is rather artificial and as a whole it is meaningful only in a theoretical sense: the odd extension of the concentration function beyond the $(-a, 0)$ domain implies that in the $(0,a)$ region we have negative concentration values. The restriction from $\Omega$ to $\Omega_0$ however remains completely meaningful in a physical sense as well.
\end{remark}

Now that we have constructed a well-defined periodic domain, we are ready to set the initial conditions for this framework. As this section focuses on strong solutions, we naturally increase the regularity of the initial data: specifically, we assume \begin{equation}\label{stronginitdataregularities}{\mathbf{u}_h}_0 \in H^3(\Omega), {u_3}_0 \in H^2(\Omega) \text{ with } \nabla \cdot \mathbf{u}_0 = 0, \text{ and } c_0 \in H^2 (\Omega).\end{equation} This is of key importance, because using classical results of the framework of strong solutions, (\ref{stronginitdataregularities}) guarantees the existence of local strong solutions $\mathbf{u}^\epsilon$ and $c^\epsilon$ for the anistropic system, which is fundamental for our results.

Our final step is the specification of the source terms, which in this case involves a convolution with a smooth function in order to have a regular source fitting the strong solution scheme.

Let $\delta_\epsilon$ be as specified in (\ref{delta_eps_source_def}) and let $\varphi$ be a $C^\infty_0 (\Omega)$ smooth function. We define the anisotropic and hydrostatic source terms respectively as the following convolutions:

\begin{equation} \label{sourcetermdefinitionstrong} s^\epsilon = \delta^\epsilon \ast \varphi, \quad s = \delta \ast \varphi.\end{equation}

\subsection{Main strong convergence theorem}

The main result of this section describing the strong solutions is analogous to Theorem \ref{maintheoremWEAK}; with the two major differences being --- apart from the quality of the solutions themselves --- the change to a periodic domain, and the assumption of an $H^3$ and $H^2$ initial data for the horizontal velocity and the concentration, respectively. Before stating the main convergence theorem of this section, we first describe a hydrostatic existence result that is essential for the fact of the convergence.

\begin{theorem} \label{existenceconcentrationSTRONG}
Let ${\mathbf{u}_h}_0 \in H^3(\Omega),$ ${u_3}_0 \in H^2(\Omega)$ with $\nabla \cdot \mathbf{u}_0 = 0$, let $c_0 \in H^2 (\Omega),$ and let's assume that the  conditions (\ref{HNSC_bc-strong}) hold. Then there exists a unique global strong solution $(\mathbf{u}, c)$ of the downwind-matching hydrostatic polluted atmosphere limit model (\ref{HNSC1}) - (\ref{HNSC4}), (\ref{HNSC5_strong_cshift}), (\ref{HNSC_init}). Specifically, the concentration function $c$ satisfies
\begin{equation} \label{regularity_c_first}
c \in L^\infty([0, \infty), H^2(\Omega)) \text{ and } \nabla_d c \in L^2_{loc}([0, \infty), H^2(\Omega)),
\end{equation}
and
\begin{equation}
\partial_t c \in L^2_{loc}([0, \infty), H^1(\Omega)),
\end{equation}
while the horizontal velocity strong solution $\mathbf{u}_h$ has the regularities
\begin{equation} \label{u_regularities_first}
\mathbf{u}_h \in L^\infty([0, \infty), H^3(\Omega)) \cap L^2_{loc}([0, \infty), H^4(\Omega))
\end{equation}
and
\begin{equation} \label{u_regularities_second}
\partial_t \mathbf{u}_h \in L^2_{loc}([0, \infty), H^2(\Omega)).
\end{equation}
\end{theorem}
We provide the proof of Theorem \ref{existenceconcentrationSTRONG} in the appendix of this article. Now we are ready to describe this section's main theorem.

\begin{theorem} \label{maintheoremSTRONG}
Let ${\mathbf{u}_h}_0 \in H^3(\Omega),$ ${u_3}_0 \in H^2(\Omega)$ with $\int_\Omega {\mathbf{u}_h}_0 \diff x \diff y \diff z = 0$ and $\nabla \cdot \mathbf{u}_0 = 0$, and let $c_0 \in H^2 (\Omega).$ There is a positive $\epsilon_0$ threshold such that for any $\epsilon < \epsilon_0,$ the strong solution $(\mathbf{u}^\epsilon, c^\epsilon)$ of the anisotropic equations (\ref{ANSC1}) - (\ref{ANSC4}), (\ref{ANSC5_strong_cshift}), (\ref{ANSC_init}), (\ref{ANSC_bc-strong}) exists globally in time and as the aspect ratio $\epsilon$ tends to zero, it converges strongly to the unique global strong solution $(\mathbf{u}, c)$ of the hydrostatic equations of the polluted atmosphere (\ref{HNSC1}) - (\ref{HNSC4}), (\ref{HNSC5_strong_cshift}), (\ref{HNSC_init}), (\ref{HNSC_bc-strong}). In more detail, the following strong convergence results hold:

\begin{align} 
& (\mathbf{u}_h^\epsilon, \epsilon u_3^\epsilon) \to (\mathbf{u}_h, 0) &\quad & \text{in $L_t^\infty H_x^2$,} \label{StrongConvVelocity1statement} \\
& (\nabla \mathbf{u}_h^\epsilon, \epsilon \nabla u_3^\epsilon, u_3^\epsilon) \to (\nabla \mathbf{u}_h, 0, u_3) &\quad & \text{in $L_t^2 H_x^2$,} \label{StrongConvVelocity2statement} \\ 
& (u_3^\epsilon, c^\epsilon) \to (u_3, c) &\quad & \text{in $L_t^\infty H_x^1,$} \label{StrongConvVelocity3andConcentrationstatement} \\
& \nabla_d c^\epsilon \to \nabla_d c &\quad & \text{in $L_t^2 H_x^1.$} \label{StrongConvConcentrationstatement}
\end{align}

\end{theorem}

\begin{proof}
The first part of the proof concerns exclusively the velocity functions, and it follows easily from the achievements of \cite{li2019primitive} and \cite{petcu2005sobolev}. Regarding specific details we refer to these above mentioned papers; here we collect only the most important cornerstones of the calculations regarding $\mathbf{u}^\epsilon$ and $\mathbf{u}$ for completeness.
\begin{enumerate}
\item There exist $\mathbf{u}^\epsilon$ and $\mathbf{u}$ unique global strong solutions for the anisotropic and hydrostatic velocity systems respectively, assuming $\epsilon < \epsilon_0$ and ${\mathbf{u}_h}_0 \in H^3(\Omega)$, where  $\epsilon_0$ is a constant defined in \cite{li2019primitive} depending only on $\norm{{\mathbf{u}_h}_0}_{H^3}$ and the domain itself.
Specifically, for ${\mathbf{u}_h}_0 \in H^3(\Omega),$ the hydrostatic horizontal velocity strong solution $\mathbf{u}_h$ satisfies
\begin{equation} \label{u_regularities_first}
\mathbf{u}_h \in L^\infty([0, \infty), H^3(\Omega)) \cap L^2_{loc}([0, \infty), H^4(\Omega))
\end{equation}
and
\begin{equation} \label{u_regularities_second}
\partial_t \mathbf{u}_h \in L^2_{loc}([0, \infty), H^2(\Omega)).
\end{equation}
\item We also mention that for an ${\mathbf{u}_h}_0 \in H^3(\Omega)$ initial data the velocity difference functions (\ref{velocitydifferencefunctions}) satisfy
\begin{equation} \label{UdifferenceEstimate}
\sup_{0 \leq t < \infty}{\norm{(\mathbf{U}_h^\epsilon, \epsilon U_3^\epsilon)}_{H^2}^2} + \int_0^\infty \norm{\nabla (\mathbf{U}_h^\epsilon, \epsilon U_3^\epsilon)}_{H^2}^2 \leq \kappa \epsilon^2.
\end{equation}

\end{enumerate}

As a consequence, the following strong convergences hold for an ${\mathbf{u}_h}_0 \in H^3(\Omega)$ initial data:

\begin{align}
& (\mathbf{u}_h^\epsilon, \epsilon u_3^\epsilon) \to (\mathbf{u}_h, 0) &\quad & \text{strongly in $L_t^\infty H_x^2$,} \label{StrongConvVelocity1} \\
& (\nabla \mathbf{u}_h^\epsilon, \epsilon \nabla u_3^\epsilon, u_3^\epsilon) \to (\nabla \mathbf{u}_h, 0, u_3) &\quad & \text{strongly in $L_t^2 H_x^2$,} \label{StrongConvVelocity2} \\ 
& u_3^\epsilon \to u_3 &\quad & \text{strongly in $L_t^\infty H_x^1.$} \label{StrongConvVelocity3}
\end{align}

Theorem \ref{maintheoremSTRONG}'s statement on the strong convergence of the velocity functions is obvious with the above result.

In the second part of the proof we present the computations verifying the strong convergence of the concentration function $c^\epsilon.$ The structure of these calculations heavily depends on the existence of both the anisotropic concentration $c^\epsilon$ and the limit solution $c.$ 
%The existence of the solution function $c^\epsilon$ follows from analogous arguments to those guaranteeing the existence of $\mathbf{u}^\epsilon$ using the initial regularities (\ref{stronginitdataregularities}). On the other hand, the existence result concerning the hydrostatic downwind-matching solution $c$ is given by Theorem \ref{existenceconcentrationSTRONG}.

The focus of the proof's remaining part is to construct an estimate for the concentration difference function and show that
\begin{equation}
c^\epsilon \to c \text{ in } L_t^\infty H_x^1 \text{ and } \nabla_d c^\epsilon \to \nabla_d c \text{ in } L_t^2 H_x^1.
\end{equation}

Analogously to the $\mathbf{U}_h^\epsilon, U_3^\epsilon$ notation we introduce 
$$C^\epsilon = c^\epsilon-c \ \text{and}\  S^\epsilon = s^\epsilon-s$$ 
to denote the difference between anisotropic and hydrostatic functions. Subtracting (\ref{HNSC5_strong_cshift}) from (\ref{ANSC5_strong_cshift}) we obtain the governing equation for $C^\epsilon$
\begin{equation}
\partial_t C^\epsilon + \mathbf{u}^\epsilon \cdot \nabla c^\epsilon - \mathbf{u} \cdot \nabla c = \frac{1}{a} U_3^\epsilon + \epsilon K_1 \partial_{11} c^\epsilon + \Delta_d C^\epsilon + S^\epsilon,
\end{equation}
which can be equivalently reformulated as
%$\mathbf{u}^\epsilon \cdot \nabla c^\epsilon - \mathbf{u} \cdot \nabla c$ as $\mathbf{U}^\epsilon \nabla C^\epsilon + \mathbf{U}^\epsilon \nabla c+ \mathbf{u} \nabla C^\epsilon,$ which yields
\begin{equation} \label{governCepsilon}
\begin{split}
& \partial_t C^\epsilon + \mathbf{U}^\epsilon \nabla C^\epsilon + \mathbf{U}^\epsilon \nabla c+ \mathbf{u} \nabla C^\epsilon \\
& = \frac{1}{a} U_3^\epsilon + \epsilon K_1 \partial_{11} C^\epsilon + \epsilon K_1 \partial_{11} c + \Delta_d C^\epsilon + S^\epsilon.
\end{split}
\end{equation}

Multiplying (\ref{governCepsilon}) by $- \partial_{ii} C^\epsilon, i=1,2,3$ we obtain
\begin{equation}
\begin{split}
& \frac{1}{2}\frac{\diff}{\diff t} \norm{\nabla C^\epsilon}_2^2 + \epsilon K_1 \norm{\partial_{11} C^\epsilon}_2^2 + K_2 \norm{\partial_{22} C^\epsilon}_2^2 + K_3 \norm{\partial_{33} C^\epsilon}_2^2 + \sum_{i \neq j} C_{ij} \norm{\partial_{ij} C^\epsilon}_2^2 \\
& = \sum_{i = 1,2,3} \int_\Omega \mathbf{U}^\epsilon \nabla C^\epsilon \partial_{ii} C^\epsilon + \int_\Omega \mathbf{U}^\epsilon \nabla c \partial_{ii} C^\epsilon + \int_\Omega \mathbf{u} \nabla C^\epsilon \partial_{ii} C^\epsilon - \int_\Omega \epsilon K_1 \partial_{11} c \partial_{ii} C^\epsilon \\
& - \sum_{i = 1,2,3} \int_\Omega S^\epsilon \partial_{ii} C^\epsilon - \frac{1}{a} \int_\Omega U_3^\epsilon \partial_{ii} C^\epsilon \\
& = \sum_{i = 1,2,3} I_1(i) + I_2(i) + I_3(i) + I_4(i) + I_5(i) + I_6(i),
\end{split}
\end{equation}
for any $1 \leq j \leq 3.$

In the following we estimate the terms on the right-hand side one by one. The separation to indices $i = 1,2,3$ and the following individual construction of the upper bounds is necessary since the full, $\epsilon$-free Laplacian is missing from the left-hand side, and this forces us to change the standard uniform approach and consider some of the integral terms on their own.\\[0.5cm]
\textbf{Estimate for $I_1(i), i = 1,2,3.$}

We begin by taking $i=1.$
\begin{equation}
\begin{split}
& I_1(1) = \int_\Omega \mathbf{U}^\epsilon \nabla C^\epsilon \partial_{11} C^\epsilon \\
& = \int_\Omega U_1^\epsilon \partial_1 C^\epsilon \partial_{11} C^\epsilon + \int_\Omega U_2^\epsilon \partial_2 C^\epsilon \partial_{11} C^\epsilon + \int_\Omega U_3^\epsilon \partial_3 C^\epsilon \partial_{11} C^\epsilon \\
& = I_1(1)(a) + I_1(1)(b) + I_1(1)(c).
\end{split}
\end{equation}
Firstly we estimate $I_1(1)(a).$ Integration by parts  and H\"older inequality gives
%
%$$\int_\Omega U_1^\epsilon \partial_1 C^\epsilon \partial_{11} C^\epsilon = - \frac{1}{2} \int_\Omega \partial_1 U_1^\epsilon (\partial_1 C^\epsilon)^2,$$
%
%thus using the  we obtain
\begin{equation}
\begin{split}
& I_1(1)(a) = \int_\Omega U_1^\epsilon \partial_1 C^\epsilon \partial_{11} C^\epsilon =- \frac{1}{2} \int_\Omega \partial_1 U_1^\epsilon (\partial_1 C^\epsilon)^2\\
& \leq C \norm{\partial_1 U_1^\epsilon}_{L^\infty} \cdot \norm{\partial_1 C^\epsilon}_2^2 \leq C \norm{\partial_1 U_1^\epsilon}_{H^2} \cdot \norm{\nabla C^\epsilon}_2^2 \\
& \leq C \norm{U_1^\epsilon}_{H^3} \cdot \norm{\nabla C^\epsilon}_2^2.
\end{split}
\end{equation}

The above calculation is one of the points throughout the proof where the $H^3$ regularity is required sharply, otherwise the closure is not possible (note that this required regularity is guaranteed by (\ref{UdifferenceEstimate})).

We can estimate $I_1(1)(b), I_1(1)(c)$ in one common step. For $j = 2,3$ we have
$$ \int_\Omega U_j^\epsilon \partial_j C^\epsilon \partial_{11} C^\epsilon = - \int_\Omega \partial_1 U_j^\epsilon \partial_j C^\epsilon \partial_1 C^\epsilon - \int_\Omega U_j^\epsilon \partial_{j1} C^\epsilon \partial_1 C^\epsilon.$$
Firstly, by the H\"older, Ladyzhenskaya and Young inequalities we have
\begin{equation}
\begin{split}
& \int_\Omega \partial_1 U_j^\epsilon \partial_j C^\epsilon \partial_1 C^\epsilon \leq \| \partial_1 U_j^\epsilon \| _3 \norm{\partial_j C^\epsilon}_6 \norm{\partial_1 C^\epsilon}_2 \\
& \leq \| \partial_1 U_j^\epsilon \|_2^{1/2} \|\nabla \partial_1 U_j^\epsilon \|_2^{1/2} \norm{\nabla \partial_j C^\epsilon}_2 \norm{\partial_1 C^\epsilon}_2 \\
& \leq \zeta \norm{\nabla \partial_j C^\epsilon}_2^2 + \frac{1}{\zeta} \| \partial_1 U_j^\epsilon \|_2 \|\nabla \partial_1 U_j^\epsilon \|_2 \norm{\partial_1 C^\epsilon}_2^2 \\
& \leq \zeta \norm{\nabla \partial_j C^\epsilon}_2^2 + \frac{1}{\zeta} \norm{\mathbf{U}_h^\epsilon}_{H^3}^2 \norm{\nabla C^\epsilon}_2^2.
\end{split}
\end{equation}

Note that for $j=3$ this approach --- which is different from what we applied for $I_1(1)(a)$ --- is necessary, since using the analogous estimate $\| \partial_1 U_3^\epsilon \| _ {L^\infty} \leq \norm{U_3^\epsilon}_{H^3} \leq \| \mathbf{U}_h^\epsilon \| _{H^4}$ arrives to an excessively high level of $H^k$ space for which we do not have an estimate anymore, considering we only have an $H^3$ initial data. Since in the $\int_\Omega U_j^\epsilon \partial_{j1} C^\epsilon \partial_1 C^\epsilon$ second term we do not have a derivative on the velocity function, we can simply use
\begin{equation} \nonumber
\begin{split}
& \int_\Omega U_j^\epsilon \partial_{j1} C^\epsilon \partial_1 C^\epsilon \leq \| U_j^\epsilon \| _{L^\infty} \norm{\partial_{j1} C^\epsilon}_2 \norm{\partial_1 C^\epsilon}_2 \\
& \leq \zeta \norm{\partial_{j1} C^\epsilon}_2^2 + \frac{1}{\zeta} \| U_j^\epsilon \| _{L^\infty}^2 \norm{\nabla C^\epsilon}_2^2 \leq \zeta \norm{\partial_{j1} C^\epsilon}_2^2 + \frac{1}{\zeta} \| \mathbf{U}_h^\epsilon \| _{H^3}^2 \norm{\nabla C^\epsilon}_2^2. 
\end{split}
\end{equation}
Altogether for $j = 2,3$ we have
\begin{equation}
\begin{split}
& I_1(1)(b) + I_1(1)(c) = \sum_{j=2}^3 \int_\Omega U_j^\epsilon \partial_j C^\epsilon \partial_{11} C^\epsilon \\
& \leq C \sum_{j=2}^3 \zeta \norm{\nabla \partial_j C^\epsilon}_2^2 + C \frac{1}{\zeta} \norm{\mathbf{U}_h}_{H^3}^2 \norm{\nabla C^\epsilon}_2^2.
\end{split}
\end{equation}
Finally we estimate $I_1(i)$ for $i = 2,3$ as follows
\begin{equation}
\begin{split}
& I_1(2) + I_1(3) = \sum_{i = 2,3} \int_\Omega \mathbf{U}^\epsilon \nabla C^\epsilon \partial_{ii} C^\epsilon \\
& \leq \sum_{i = 2,3} \| \mathbf{U}^\epsilon \| _{L^\infty} \norm{\nabla C^\epsilon}_2 \norm{\partial_{ii}C^\epsilon}_2 \lesssim \sum_{i = 2,3} \zeta \norm{\partial_{ii} C^\epsilon}_2^2 + \frac{1}{\zeta} \| \mathbf{U}_h^\epsilon \| _ {H^3} ^ 2 \norm{\nabla C^\epsilon}_2^2. \raisetag{50pt}
\end{split}
\end{equation}
\textbf{Estimate for $I_2(i), i=1,2,3.$}

By using the H\"older and Young inequalities, for $i = 1$ we obtain
\begin{equation} \label{cH2normappears}
\begin{split}
& I_2(1) = \int_\Omega \mathbf{U}^\epsilon \nabla c \partial_{11} C^\epsilon = -\int_\Omega \partial_1 \mathbf{U}^\epsilon \nabla c \partial_1 C^\epsilon - \int_\Omega \mathbf{U}^\epsilon \nabla \partial_1 c \partial_1 C^\epsilon \\
& \lesssim \| \partial_1 \mathbf{U}^\epsilon \|_3 \norm{\nabla c}_6 \norm{\partial_1 C^\epsilon}_2 + \| \mathbf{U}^\epsilon \|_{L^\infty} \norm{\nabla \partial_1 c}_2 \norm{\partial_1 C^\epsilon}_2 \\
& \lesssim \| \nabla \partial_1 \mathbf{U}^\epsilon \|_2 \norm{c}_{H^2} \norm{\partial_1 C^\epsilon}_2 + \| \mathbf{U}^\epsilon \|_{H^2} \norm{c}_{H^2} \norm{\partial_1 C^\epsilon}_2 \\
& \lesssim \| \mathbf{U}_h^\epsilon \|_{H^3} \norm{c}_{H^2} \norm{\partial_1 C^\epsilon}_2 \lesssim \| \mathbf{U}_h^\epsilon \|_{H^3}^2 + \norm{c}_{H^2}^2 \norm{\nabla C^\epsilon}_2^2,
\end{split}
\end{equation}
while $I_2(i), i = 2,3$ can be estimated by
\begin{equation}
\begin{split}
& I_2(2) + I_2(3) = \sum_{i=2,3} \int_\Omega \mathbf{U}^\epsilon \nabla c \partial_{ii} C^\epsilon \leq \sum_{i=2,3} \norm{\mathbf{U}^\epsilon}_{L^\infty} \norm{\nabla c}_2 \norm{\partial_{ii} C^\epsilon}_2 \\
& \leq \sum_{i=2,3} \zeta \norm{\partial_{ii} C^\epsilon}_2^2 + \kappa \frac{1}{\zeta} \norm{\mathbf{U}_h^\epsilon}_{H^3}^2 \norm{\nabla c}_2^2.
\end{split}
\end{equation}
Note that the regularities (\ref{regularity_c_first}) listed in Theorem \ref{existenceconcentrationSTRONG} guarantee that the $\norm{c}_{H^2}^2$ term in (\ref{cH2normappears}) is bounded.\\[0.5cm]
\textbf{Estimate for $I_3(i), i=1,2,3.$}

Due to their structure, computations for $\int_\Omega u \nabla C^\epsilon \partial_{ii} C^\epsilon$ are analogous to those used to estimate $ \int_\Omega U^\epsilon \nabla C^\epsilon \partial_{ii} C^\epsilon.$

Using the same tools, i.e. the H\"older, Ladyzhenskaya and Young inequalities, for $i=1$ we obtain
\begin{equation}
I_3(1) \lesssim \norm{u_1}_{H^3} \norm{\nabla C^\epsilon}_2^2 + \sum_{j=2,3} \zeta \norm{\nabla \partial_j C^\epsilon}_2^2 + \frac{1}{\zeta} \norm{\mathbf{u}_h}_{H^3}^2 \norm{\nabla C^\epsilon}_2^2.
\end{equation}
On the other hand, for $i = 2,3$ we have
\begin{equation}
I_3(2) + I_3(3) \lesssim \zeta \sum_{i=2,3} \norm{\partial_{ii} C^\epsilon}_2^2 + \frac{1}{\zeta} \| \mathbf{u}_h \| _ {H^3} ^ 2 \norm{\nabla C^\epsilon}_2^2.
\end{equation}

Note that thanks to (\ref{u_regularities_first}) the $\norm{\nabla C^\epsilon}_2^2$ term is multiplied by a quantity that we can estimate well.\\[0.5cm]
\textbf{Estimate for $I_4(i), i = 1,2,3.$}

In this case we do not need to separate the case $i=1$ from $i=2,3,$ and we can use the general estimate
\begin{equation}
\begin{split}
& I_4(i) = \epsilon K_1 \int_\Omega \partial_{11} c \partial_{ii} C^\epsilon \lesssim \epsilon \norm{\partial_{11} c}_2 \norm{\partial_{ii} C^\epsilon}_2 \\
& \lesssim \epsilon \zeta \norm{\partial_{ii} C^\epsilon}_2^2 + \frac{1}{\zeta}\epsilon \norm{\partial_{11} c}_2^2
\end{split}
\end{equation}
for all $i =1,2,3.$ Since here we do have an $\epsilon$ in front of the Laplacian $ \norm{\partial_{ii} C^\epsilon}_2^2,$ it can be merged into the left-hand side later even for $i = 1.$\\[0.5cm]
\textbf{Estimate for $I_5 (i), i = 1,2,3.$}

For all $i = 1,2,3$ values we have
\begin{equation}
\begin{split}
& - \int_\Omega S^\epsilon \partial_{ii} C^\epsilon = - \int_\Omega (s^\epsilon - s) \partial_{ii} C^\epsilon = - \int_\Omega \big ( (\delta^\epsilon - \delta ) \ast \varphi \big ) \cdot \partial_{ii} C^\epsilon \\
& = \int_\Omega \partial_i \big ( (\delta^\epsilon - \delta ) \ast \varphi \big ) \cdot \partial_i C^\epsilon \diff x \\
& \lesssim \int_\Omega \bigg( \partial_i \big ( (\delta^\epsilon - \delta ) \ast \varphi \big ) \bigg ) ^2 \diff x + \int_\Omega (\partial_i C^\epsilon)^2 \diff x \\
& \lesssim \int_\Omega \bigg ( (\delta^\epsilon - \delta ) \ast \partial_i \varphi \bigg ) ^2 \diff x + \norm{\nabla C^\epsilon}_2^2 \lesssim \epsilon^2 + \norm{\nabla C^\epsilon}_2^2,
\end{split}
\end{equation}
where we have used the properties of $\delta^\epsilon$ and $\delta.$\\[0.5cm]
\textbf{Estimate for $I_6 (i), i = 1,2,3.$}

Integration by parts combined with the H\"older and Young inequalities easily gives
\begin{equation}
\begin{split}
& - \frac{1}{a} \int_\Omega U_3^\epsilon \partial_{ii} C^\epsilon \lesssim \norm{\nabla U_3^\epsilon}_2 \norm{\nabla C^\epsilon}_2 \lesssim \norm{\mathbf{U}_h^\epsilon}_{H^2}^2 + \norm{\nabla C^\epsilon}_2^2.
\end{split}
\end{equation}

\bigskip

Summarising the previously obtained six estimates for $I_1$ -- $I_6$ and merging the second order concentration terms into the left-hand side (note that it is indeed possible to do so as we have constructed the bounds in a way that $\partial_{11} C^\epsilon$ is only present with an $\epsilon$ multiplier), we eventually have
\begin{equation}
\begin{split}
& \frac{1}{2} \frac{\diff}{\diff t} \norm{\nabla C^\epsilon}_2^2 + \epsilon K_1 \norm{\partial_{11} C^\epsilon}_2^2 + K_2 \norm{\partial_{22} C^\epsilon}_2^2 + K_3 \norm{\partial_{33} C^\epsilon}_2^2 + \sum_{i \neq j} C_{ij} \norm{\partial_{ij} C^\epsilon}_2^2 \\
& \lesssim (\norm{\mathbf{u}_h}_{H^3} + \norm{\mathbf{u}_h}_{H^3}^2 + \norm{\mathbf{U}_h^\epsilon}_{H^3} + \norm{\mathbf{U}_h^\epsilon}_{H^3}^2 + 1 + \norm{c}_{H^2}^2) \norm{\nabla C^\epsilon}_2^2 \\
& + \norm{\mathbf{U}_h^\epsilon}_{H^3}^2 + \norm{\mathbf{U}_h^\epsilon}_{H^3}^2 \norm{\nabla c}_2^2 + \epsilon K_1 \norm{\partial_{11} c}_2^2 + \epsilon^2 \\
& = A(t) \norm{\nabla C^\epsilon}_2^2 + B(t) \raisetag{40pt}
\end{split}
\end{equation}
for any $1 \leq j \leq 3.$

Now taking into account (\ref{regularity_c_first}), (\ref{u_regularities_first}) and (\ref{UdifferenceEstimate}), we have that
\[ \int_0^T B(t) \diff \tau = \kappa \int_0^T \epsilon \norm{\partial_{11} c}_2^2 + \norm{\mathbf{U}_h^\epsilon}_{H^3}^2 \norm{\nabla c}_2^2 + \norm{\mathbf{U}_h^\epsilon}_{H^3}^2 + \epsilon^2 \diff \tau \lesssim \epsilon + \epsilon^2 + \epsilon^2 T,\]
and
\begin{equation}
\begin{split}
& \int_0^T A(t) \diff \tau = \kappa \int_0^T \norm{\mathbf{U}_h^\epsilon}_{H^3} + \norm{\mathbf{u}_h}_{H^3} + \norm{\mathbf{U}_h^\epsilon}_{H^3}^2 + \norm{\mathbf{u}_h}_{H^3}^2 + \norm{c}_{H^2}^2 + 1 \diff \tau \\
& \lesssim \sqrt{T} + 1 + T.
\end{split}
\end{equation}
Keeping in mind that $C^\epsilon(0) = 0,$ eventually the Gr\"onwall inequality yields
\begin{equation} \label{finalboundforCepsilon}
\begin{split}
& \norm{C^\epsilon}_{L^\infty(0,T; H^1(\Omega))} + \norm{\nabla_d C^\epsilon}_{L^2(0,T; H^1(\Omega))} \\
& \lesssim (\epsilon + \epsilon^2 \cdot T) \exp{\kappa (\sqrt{T} + 1 + T)}.
\end{split}
\end{equation}

Here we highlight an important structural difference in the (\ref{finalboundforCepsilon}) final estimate compared to that in \cite{li2019primitive} (for example Proposition 5.3): the time uniformity of the bound is not preserved for our extended, concentration-including version of the modelling equations. The estimate provides a meaningful bound for any arbitrarily big finite point in time, but the estimate's time-dependence also means that the convergence gets slower as time evolves. Note also that (\ref{finalboundforCepsilon}) implies the global existence of $c^\epsilon,$ as it suggests that $C^\epsilon$ can be continued from any arbitrary finite $T^* > 0$ time.

Finally we conclude that for any given time $T > 0$ we have the global strong convergence results
\begin{equation} \label{finalstrongconvresultC}
c^\epsilon \to c \text{ in } L_t^\infty H_x^1 \text{ and } \nabla_d c^\epsilon \to \nabla_d c \text{ in } L_t^2 H_x^1,
\end{equation}
moreover the rate of convergence is $\epsilon.$

Combining the velocity convergence results (\ref{StrongConvVelocity1}) - (\ref{StrongConvVelocity3}) with the (\ref{finalstrongconvresultC}) concentration convergence obtained above completes the proof.

\end{proof}

\appendix

\section{Existence of the strong solutions of the hydrostatic equations of the polluted atmosphere}

The appendix is dedicated to showing the existence of the unique global strong solution $(\mathbf{u}_h, u_3, c)$ of the hydrostatic limit system (\ref{HNSC1}) - (\ref{HNSC4}), (\ref{HNSC5_strong_cshift}), (\ref{HNSC_init}), (\ref{HNSC_bc-strong}) --- in other words we prove Theorem \ref{existenceconcentrationSTRONG}. The proof's first part strongly uses the existence result achieved in \cite{cao2014local}: here the authors focus on a system that is very similar to ours, including an incomplete Laplacian in the convection-diffusion equation. However, while they work with a temperature function $T,$ in our case the additional function to the velocity-pressure pair is the concentration $c.$ Even though there is a strong analogy between the equations describing the dynamics of concentration and temperature, yet the nature of how each of these functions are coupled with the velocity equations is different, and overall this contributes to a very significant structural divergence. Specifically, the difference between the equations $$ \partial_3 p + T = 0$$ and $$ \partial_3 p = 0$$ turns out to be a cardinal one concerning global existence. Observe in these equations that for the case of the temperature there is a reciprocal coupling between the momentum equations and the convection-diffusion equation, on the other hand the coupling for the pollution concentration is one-sided, the hydrostatic equation does not depend on the function $c.$

The method of \cite{cao2014local} is easily applicable with minor changes to our system until reaching the verification of local existence; the transplantation of their global existence proof on the other hand becomes very problematic. The presence of $T$ in the hydrostatic equation makes certain cancellations possible in their case which eventually guarantee the ability to close estimates that otherwise would not lead to meaningful implications. Since our hydrostatic equation does not include $c,$ we will work out a set of rather independent estimates to show the continuability of the concentration function.

\begin{proof}[Proof of Theorem \ref{existenceconcentrationSTRONG}]

The proof is built up by two main steps. We first verify the local existence of the limit solutions following the ideas of \cite{cao2014local}, while in the second step we show their global existence exploiting the already known (see (\ref{u_regularities_first}), (\ref{u_regularities_second})) regularities of the velocity functions: this means that we do not need to provide bounds simultaneously for the velocity and concentration functions, and in this part we can focus exclusively on estimating $c.$

\bigskip

%%% LOCAL EXISTENCE %%%
\textbf{Part 1: Local existence}

As far as local existence is concerned, the proof given in sections 2 and 3 in \cite{cao2014local} can be easily customised to fit our system. In order to transplant the proof to the case of the polluted atmosphere, we need to check the robustness of its structure with respect to the given form of the hydrostatic equation. The necessary modifications are small-scale and rather obvious, therefore we do not unfold the entire proof here, however we do outline the main steps of the verification process:

\begin{enumerate}

\item We apply a commonly used idea regarding the formal elimination of the vertical velocity by expressing it through the divergence-free condition:

\begin{equation} \label{u_3_expressed_div_free}
u_3 = - \int_{-a}^z \nabla_h \cdot \mathbf{u}_h (x, y, \xi, t) \diff \xi.
\end{equation}

On the other hand, we define a modified version of the convection-diffusion equation (\ref{HNSC5_strong_cshift}) by completing the missing term in the Laplacian adding the term $\mu \partial_{11} c$ with $\mu > 0$; specifically, the set of equations takes the form
\begin{equation} \label{existenceeqs1}
\begin{split}
\partial_t \mathbf{u}_h + \Delta_\nu \mathbf{u}_h + & (\mathbf{u}_h \cdot \nabla_h) \mathbf{u}_h - \bigg ( \int_{-a}^z \nabla_h \cdot \mathbf{u}_h (x, y, \xi, t) \diff \xi \bigg ) \partial_3 \mathbf{u}_h \\
& + f_0 k \times \mathbf{u}_h + \nabla_h p_s (x,y,t) = 0,
\end{split}
\end{equation}
\begin{equation} \label{existenceeqs2}
\nabla_h \cdot \bar{\mathbf{u}}_h = 0,
\end{equation}
\begin{equation} \label{existenceeqs3}
\begin{split}
\partial_t c+ \mathbf{u}_h \cdot \nabla_h c + \bigg ( - & \int_{-a}^z \nabla_h \cdot \mathbf{u}_h (x, y, \xi, t) \diff \xi \bigg ) \cdot \bigg (\partial_3 c - \frac{1}{a} \bigg ) \\
& - \Delta_{d} c - \mu \partial_{11} c = 0.
\end{split}
\end{equation}

Here $\bar{\mathbf{u}}_h$ is the vertically averaged version of the function, i.e. $$\frac{1}{2a} \int_{-a}^a \mathbf{u}_h (x,y,z) \diff z,$$ and $p_s$ is the classical notation for a pressure function depending on only the horizontal variables. Fore more details of this standard reformulation process see for example \cite{cao2014local} and \cite{cao2007global}.

\item Using the contractive mapping principle, we can prove that for any fix positive $\mu$ value there exist local strong solutions $\mathbf{u}_h^\mu$ and $c^\mu$ to the modified system (\ref{existenceeqs1}) - (\ref{existenceeqs3}) on $\Omega \times (0, t_\mu)$ such that

\[ (\mathbf{u}_h^\mu, c^\mu) \in L^2 (0, t^\mu; H^3(\Omega)), \quad (\partial_t \mathbf{u}_h^\mu, \partial_t c^\mu) \in L^2 (0, t^\mu; H^1 (\Omega)), \]
where $t^\mu$ depends only on $\Omega, \mu, $ the initial data, and the viscosity and diffusivity constants.

\item We derive the existence of the unique local strong solution $\mathbf{u}_h, c$ of the limit system as the strong limit of the sequences $\mathbf{u}_h^\mu, c^\mu,$ having the parameter $\mu$ tend to zero. In more detail, we establish uniform in $\mu$ estimates for $\mathbf{u}_h^\mu, c^\mu$ and we construct $\mu-$ independent lower bounds for the existence time $t_0$. These estimates combined with a version of the Aubin-Lions lemma yield

\begin{equation}
(\mathbf{u}_h^{\mu_j}, c^{\mu_j}) \to (\mathbf{u}_h, c) \text{ in } C(0, t_0; H^1(\Omega)),
\end{equation}
and
\begin{equation}
\mathbf{u}_h^{\mu_j} \to \mathbf{u}_h \text{ in } L^2(0, t_0; H^2(\Omega)), \quad \nabla_d c^{\mu_j} \to \nabla_d c \text{ in } L^2(0, t_0; H^1(\Omega)).
\end{equation}

The uniqueness follows from standard steps estimating the difference functions and using the Gr\"onwall inequality.

\end{enumerate}

\bigskip

%%% GLOBAL EXISTENCE %%%
\textbf{Part 2: Global existence}

As we mentioned in the introductory part of this proof, here we only need to provide estimates verifying the continuability of $c,$ as the global existence of the velocity function immediately follows from the results of \cite{li2019primitive} and \cite{petcu2005sobolev}.

\bigskip

\textbf{$\norm{\Delta c}_2$ and $\norm{\Delta \nabla_d c}_2$ estimates}

Let's take an arbitrary positive finite $T$ time. Plugging the identity (\ref{u_3_expressed_div_free}) into the convection-diffusion equation (\ref{HNSC5_strong_cshift}), multiplying it by $\Delta^2 c,$ and integrating the result on $\Omega$ gives
\begin{equation} \label{firstformboundc}
\begin{split}
& \frac{1}{2} \frac{\diff}{\diff t} \norm{\Delta c}_2^2 + C \norm{\Delta \nabla_d c}_2^2 \\
& \leq \int_\Omega \Delta s \Delta c - \Delta ( \mathbf{u}_h \cdot \nabla_h c) \cdot \Delta c + \Delta \bigg ( \bigg ( \int_{-a}^z \nabla_h \cdot \mathbf{u}_h \diff \xi \bigg) \bigg ( \partial_3 c - \frac{1}{a} \bigg) \bigg ) \cdot \Delta c.
\end{split}
\end{equation}

In the following we construct an estimate for each term on the right-hand side. Firstly, we have
\begin{equation} \label{term1estimateGE}
\int_\Omega \Delta s \Delta c \leq \norm{\Delta s}_2^2 + \norm{\Delta c}_2^2.
\end{equation}

Note here that the $\norm{\Delta s}_2^2$ term on the right-hand is bounded because of (\ref{sourcetermdefinitionstrong}). The second term can be estimated as follows.
\begin{equation} \label{term2estimateGE}
\begin{split}
& \int_\Omega \Delta ( \mathbf{u}_h \cdot \nabla_h c) \cdot \Delta c \\
& \lesssim \int_\Omega \abs{\Delta \mathbf{u}_h} \abs{\nabla_h c} \abs{\Delta c} + \int_\Omega \abs{\nabla \mathbf{u}_h} \abs{\nabla \nabla_h c} \abs{\Delta c} + \int_\Omega \abs{\nabla_h \mathbf{u}_h} \abs{\Delta c}^2 \\
& \lesssim \norm{\mathbf{u}_h}_{H^4} \norm{\Delta c}_2^2.
\end{split}
\end{equation}

Finally the third term in the right-hand side of (\ref{firstformboundc}) can be separated as
$$\int_\Omega \Delta \bigg ( \bigg ( \int_{-a}^z \nabla_h \cdot \mathbf{u}_h \diff \xi \bigg) \bigg ( \partial_3 c - \frac{1}{a}\bigg) \bigg ) \cdot \Delta c = I_1 + I_2 + I_3,$$
and by the H\"older and Young inequalities we have
\begin{equation} \label{term3estimateGE_1}
\begin{split}
& I_1 = \int_\Omega \bigg ( \int_{-a}^z \Delta \nabla_h \cdot \mathbf{u}_h \diff \xi \bigg) \bigg ( \partial_3 c - \frac{1}{a}\bigg) \cdot \Delta c \\
& = - \frac{1}{a} \int_\Omega \bigg ( \int_{-a}^z \Delta \nabla_h \cdot \mathbf{u}_h \diff \xi \bigg) \cdot \Delta c + \int_\Omega \bigg ( \int_{-a}^z \Delta \nabla_h \cdot \mathbf{u}_h \diff \xi \bigg) \partial_3 c \cdot \Delta c \\
& = - \frac{1}{a} \int_\Omega \bigg ( \int_{-a}^z \Delta \nabla_h \cdot \mathbf{u}_h \diff \xi \bigg) \cdot \Delta c - \int_\Omega (\Delta \nabla_h \cdot \mathbf{u}_h) c \cdot \Delta c - \int_\Omega \bigg ( \int_{-a}^z \Delta \nabla_h \cdot \mathbf{u}_h \diff \xi \bigg) c \cdot \Delta \partial_3 c \\
& \lesssim \int_\Omega \bigg ( \int_{-a}^0 \abs{\Delta \nabla_h \cdot \mathbf{u}_h} \diff \xi \bigg) \abs{\Delta c} \\
& + \norm{c}_{L^\infty} \int_\Omega \abs{\Delta \nabla_h \cdot \mathbf{u}_h} \abs{\Delta c} + \norm{c}_{L^\infty} \int_\Omega \bigg ( \int_{-a}^z \abs{ \Delta \nabla_h \cdot \mathbf{u}_h} \diff \xi \bigg) \abs{\Delta \partial_3 c} \\
& \lesssim \int_M \bigg ( \int_{-a}^0 \abs{\Delta \nabla_h \cdot \mathbf{u}_h} \diff \xi \bigg) \bigg ( \int_{-a}^0 \abs{\Delta c} \diff z \bigg ) \\
& + \norm{c}_{L^\infty} \norm{\Delta \nabla_h \cdot \mathbf{u}_h}_2 \norm{\Delta c}_2 + \norm{c}_{L^\infty} \int_\Omega \bigg ( \int_{-a}^0 \abs{ \Delta \nabla_h \cdot \mathbf{u}_h} \diff \xi \bigg) \abs{\Delta \partial_3 c} \\
& \lesssim \bigg ( \int_M \bigg ( \int_{-a}^0 \abs{\Delta \nabla_h \cdot \mathbf{u}_h}^2 \diff \xi \bigg) \bigg )^{1/2} \bigg ( \int_M \bigg ( \int_{-a}^0 \abs{\Delta c}^2 \diff z \bigg ) \bigg )^{1/2} \\
& + \norm{c}_{L^\infty} (\norm{\mathbf{u}_h}_{H^3}^2 + \norm{\Delta c}_2^2) + \norm{c}_{L^\infty} \norm{\mathbf{u}_h}_{H^3} \norm{\Delta \partial_3 c}_2 \\
& \lesssim \norm{\mathbf{u}_h}_{H^3} \norm{\Delta c}_2 + \norm{c}_{L^\infty} \norm{\mathbf{u}_h}_{H^3}^2 + \norm{c}_{L^\infty} \norm{\Delta c}_2^2 + \kappa(\zeta) \norm{c}_{L^\infty}^2 \norm{\mathbf{u}_h}_{H^3}^2 + \zeta \norm{\Delta \partial_3 c}_2^2, \raisetag{60pt} \\
\end{split}
\end{equation}
where $ \zeta \norm{\Delta \partial_3 c}_2^2$ can be merged into the left-hand side of (\ref{firstformboundc}).

For the second part we obtain
\begin{equation} \label{term3estimateGE_2}
\begin{split}
& I_2 = \int_\Omega \bigg ( \int_{-a}^z \nabla \nabla_h \cdot \mathbf{u}_h \diff \xi \bigg) \nabla \partial_3 c \cdot \Delta c \\
& \lesssim \int_\Omega \bigg ( \int_{-a}^0 \abs{\nabla \nabla_h \cdot \mathbf{u}_h} \diff \xi \bigg) \abs{\nabla \partial_3 c} \abs{\Delta c} \\
& \lesssim \int_\Omega a \sup \abs{\nabla \nabla_h \cdot \mathbf{u}_h} \abs{\nabla \partial_3 c} \abs{\Delta c} \lesssim \norm{\mathbf{u}_h}_{H^4} \norm{\Delta c}_2^2,
\end{split}
\end{equation}
while for the third term we compute
\begin{equation} \label{term3estimateGE_3}
\begin{split}
& I_3 = \int_\Omega \bigg ( \int_{-a}^z \nabla_h \cdot \mathbf{u}_h \diff \xi \bigg) \Delta \partial_3 c \cdot \Delta c \leq \int_\Omega \bigg ( \int_{-a}^0 \abs{\nabla_h \cdot \mathbf{u}_h} \diff \xi \bigg) \abs{\Delta \partial_3 c} \abs{\Delta c}\\
& \lesssim \int_\Omega h \sup{\abs{\nabla_h \cdot \mathbf{u}_h}} \abs{\Delta \partial_3 c} \abs{\Delta c} \lesssim \norm{\mathbf{u}_h}_{H^3} (\zeta \norm{\Delta \partial_3 c}_2^2 + \kappa(\zeta)\norm{\Delta c}_2^2 ).
\end{split}
\end{equation}

Here the $\zeta \norm{\mathbf{u}_h}_{H^3} \norm{\Delta \partial_3 c}_2^2$ term can be merged into the left-hand side of (\ref{firstformboundc}), as $d$ contains the direction $z,$ and $\norm{\mathbf{u}_h}_{H^3}$ is finite.

Now, merging estimates (\ref{term1estimateGE}), (\ref{term2estimateGE}), (\ref{term3estimateGE_1}),(\ref{term3estimateGE_2}), (\ref{term3estimateGE_3}) into (\ref{firstformboundc}) we eventually arrive to
\begin{equation} \label{Gronwallstructureready}
\begin{split}
& \frac{1}{2} \frac{\diff}{\diff t} \norm{\Delta c}_2^2 + \kappa \norm{\Delta \nabla_d c}_2^2 \\
& \lesssim (1 + \norm{\mathbf{u}_h}_{H^3} + \norm{\mathbf{u}_h}_{H^3}^2 + \norm{\mathbf{u}_h}_{H^4} + \norm{c}_{L^\infty})\norm{\Delta c}_2^2 \\
& + \norm{\Delta s}_2^2 + \norm{c}_{L^\infty} \norm{\mathbf{u}_h}_{H^3}^2 + \norm{c}_{L^\infty}^2 \norm{\mathbf{u}_h}_{H^3}^2 + \norm{\mathbf{u}_h}_{H^3}.
\end{split}
\end{equation}

In order to bound the $\norm{c}_{L^\infty} (t)$ term, we follow Stampacchia's idea for proving the Maximum Principle. The approach is detailed in \cite{cao2012global} (see their estimate (69), the verification process in our case is analogous), which allows us to use
$$\norm{c}_{L^\infty} (t) \leq 1 + \norm{c_0}_{L^\infty} + \norm{s}_{L^\infty} t.$$

Now we are ready to utilise the Gr\"onwall inequality, provided the multiplicative term on the right-hand side of (\ref{Gronwallstructureready}) can be bounded in time as
\begin{equation} \label{gronwall_mult}
\begin{split}
& \int_0^T 1 + \norm{\mathbf{u}_h}_{H^3} + \norm{\mathbf{u}_h}_{H^3}^2 + \norm{\mathbf{u}_h}_{H^4} + \norm{c}_{L^\infty} \diff t \\
& \lesssim (T + T\norm{\mathbf{u}_h}_{L^\infty_t H^3_x} + T \norm{\mathbf{u}_h}_{L^\infty_t H^3_x}^2 + \int_0^T (1 + \norm{\mathbf{u}_h}_{H^4}^2) \diff t + \int_0^T \norm{c}_{L^\infty} \diff t )\\
& \lesssim (T + T\norm{\mathbf{u}_h}_{L^\infty_t H^3_x} + T \norm{\mathbf{u}_h}_{L^\infty_t H^3_x}^2 + \norm{\mathbf{u}_h}_{L^2_t H^4_x} + \int_0^T t \diff t )\\
& \lesssim (T + T\norm{\mathbf{u}_h}_{L^\infty_t H^3_x} + T \norm{\mathbf{u}_h}_{L^\infty_t H^3_x}^2 + \norm{\mathbf{u}_h}_{L^2_t H^4_x} + T^2) : = G_1,
\end{split}
\end{equation}
moreover we derive
\begin{equation} \label{gronwall_add}
\begin{split}
& \int_0^T \norm{\Delta s}_2^2 + \norm{c}_{L^\infty} \norm{\mathbf{u}_h}_{H^3}^2 + \norm{c}_{L^\infty}^2 \norm{\mathbf{u}_h}_{H^3}^2 + \norm{\mathbf{u}_h}_{H^3} \diff t\\
& \lesssim T \norm{s}_{H^2} + \norm{\mathbf{u}_h}_{L^\infty_t H^3_x}^2 T^2 + \norm{\mathbf{u}_h}_{L^\infty_t H^3_x}^2 T^3 + \norm{\mathbf{u}_h}_{L^\infty_t H^3_x} T := G_2.
\end{split}
\end{equation}
Finally, from (\ref{Gronwallstructureready}), (\ref{gronwall_mult}) and (\ref{gronwall_add}) we conclude
\begin{equation} \label{finalgronwallforc_part1}
\sup_{0 \leq t \leq T}\norm{\Delta c}_2^2 + \kappa \int_0^T \norm{\Delta \nabla_d c}_2^2 \leq \kappa (\norm{c_0}_{H^2} + G_2) \exp{G_1}.
\end{equation}

The estimate (\ref{finalgronwallforc_part1}) shows the overall continuability of the concentration function beyond any given finite $T$ time: considering the already well-known continuability of the velocity functions, our proof with this is concluded.

\end{proof}

\nocite{*}
\bibliographystyle{abbrv}

\begin{thebibliography}{10}

\bibitem{adams2003sobolev}
R.~A. Adams and J.~J. Fournier.
\newblock Sobolev spaces, volume 140.
\newblock {\em Academic press}, 2003.

\bibitem{ali1998partial}
I.~Ali, S.~Kalla, and H.~Khajah.
\newblock A partial differential equation related to a problem in atmospheric pollution.
\newblock {\em Mathematical and computer modelling}, 28(12):1--6, 1998.

\bibitem{anderson1995computational}
J.~D. Anderson and J.~Wendt.
\newblock Computational fluid dynamics, volume 206.
\newblock {\em Springer}, 1995.

\bibitem{andria2000mathematical}
G.~Andria, A.~Lay-Ekuakille, and M.~Notarnicola.
\newblock Mathematical models for atmospheric and industrial pollutant prediction.
\newblock In {\em XVI IMEKO World Congress, Vienna, Austria}, 2000.

\bibitem{arya2001introduction}
P.~S. Arya.
\newblock Introduction to micrometeorology, volume~79.
\newblock {\em Elsevier}, 2001.

\bibitem{azerad2001mathematical}
P.~Az{\'e}rad and F.~Guill{\'e}n.
\newblock Mathematical justification of the hydrostatic approximation in the primitive equations of geophysical fluid dynamics.
\newblock {\em SIAM Journal on Mathematical Analysis}, 33(4):847--859, 2001.

\bibitem{besson1992some}
O.~Besson and M.~Laydi.
\newblock Some estimates for the anisotropic Navier-Stokes equations and for the hydrostatic approximation.
\newblock {\em ESAIM: Mathematical Modelling and Numerical Analysis}, 26(7):855--865, 1992.

\bibitem{cao2014local}
C.~Cao, J.~Li, and E.~S. Titi.
\newblock Local and global well-posedness of strong solutions to the 3D primitive equations with vertical eddy diffusivity.
\newblock {\em Archive for Rational Mechanics and Analysis}, 214(1):35--76, 2014.

\bibitem{cao2016global}
C.~Cao, J.~Li, and E.~S. Titi.
\newblock Global well-posedness of the three-dimensional primitive equations with only horizontal viscosity and diffusion.
\newblock {\em Communications on Pure and Applied Mathematics}, 69(8):1492--1531, 2016.

\bibitem{cao2017global}
C.~Cao, J.~Li, and E.~S. Titi.
\newblock Global well-posedness of the three-dimensional primitive equations with only horizontal viscosity and diffusion.
\newblock {\em arXiv:1703.02512 [math.AP]}, 2017.

\bibitem{cao2003global}
C.~Cao and E.~S. Titi.
\newblock Global well-posedness and finite-dimensional global attractor for a 3-D planetary geostrophic viscous model.
\newblock {\em Communications on Pure and Applied Mathematics: A Journal Issued by the Courant Institute of Mathematical Sciences}, 56(2):198--233, 2003.

\bibitem{cao2007global}
C.~Cao and E.~S. Titi.
\newblock Global well-posedness of the three-dimensional viscous primitive equations of large scale ocean and atmosphere dynamics.
\newblock {\em Annals of Mathematics}, pages 245--267, 2007.

\bibitem{cao2012global}
C.~Cao and E.~S. Titi.
\newblock Global well-posedness of the 3D primitive equations with partial vertical turbulence mixing heat diffusion.
\newblock {\em Communications in Mathematical Physics}, 310(2):537--568, 2012.

\bibitem{swart2016utrecht}
H.~E. de~Swart.
\newblock The governing equations and the dominant balances of flow in the atmosphere and ocean.
\newblock {\em USPC Summerschool, Utrecht}, 2016.

\bibitem{pollatmws2019}
D.~Donatelli and N.~Juh\'asz.
\newblock Weak solution of the merged mathematical equations of the polluted atmosphere.
\newblock {\em Mathematical Methods in the Applied Sciences}, 43:9245--9261, 2020.

\bibitem{kukavica-ziane}
I.~Kukavica, and M.~Ziane.
\newblock On the regularity of the primitive equations of the ocean.
\newblock{\em Nonlinearity}, 20.12 (2007): 2739.

\bibitem{GNT2019}
H.~Gao, S.~Necasova, and T.~Tang.
\newblock{On weak-strong uniqueness and singular limit for the compressible Primitive Equations}, 
\newblock {\em  Discrete Contin. Dyn. Syst.} 40(7): 4287--4305, (2020).

\bibitem{goyal2011mathematical}
P.~Goyal and A.~Kumar.
\newblock Mathematical modeling of air pollutants: An application to Indian urban city.
\newblock In {\em Air Quality-Models and Applications}. InTech, 2011.

\bibitem{hosseini2013}
B.~Hosseini.
\newblock Dispersion of pollutants in the atmosphere: a numerical study.
\newblock Master's thesis, Simom Fraser University, 2013.

\bibitem{kathirgamanathan2002source}
P.~Kathirgamanathan, R.~McKibbin, and R.~McLachlan.
\newblock Source term estimation of pollution from an instantaneous point source.
\newblock {\em Research Letters in Information Mathematic Science, Vol. 3, pp. 59-67.}, 2002.

\bibitem{li2019primitive}
J.~Li and E.~S. Titi.
\newblock The primitive equations as the small aspect ratio limit of the Navier--Stokes equations: Rigorous justification of the hydrostatic approximation.
\newblock {\em Journal de Math{\'e}matiques Pures et Appliqu{\'e}es}, 124:30--58, 2019.

\bibitem{lions1996mathematical}
P.-L. Lions.
\newblock Mathematical Topics in Fluid Mechanics: Volume 2: Compressible Models, volume~2.
\newblock {\em Oxford University Press on Demand}, 1996.

\bibitem{marsaleix2006considerations}
P.~Marsaleix, F.~Auclair, and C.~Estournel.
\newblock Considerations on open boundary conditions for regional and coastal ocean models.
\newblock {\em Journal of Atmospheric and Oceanic Technology}, 23(11):1604--1613, 2006.

\bibitem{monin1959boundary}
A.~Monin.
\newblock On the boundary condition on the Earth surface for diffusing pollution.
\newblock In {\em Advances in Geophysics}, volume~6, pages 435--436. Elsevier, 1959.

\bibitem{pedlosky2013geophysical}
J.~Pedlosky.
\newblock Geophysical fluid dynamics.
\newblock {\em Springer Science \& Business Media}, 2013.

\bibitem{petcu2005sobolev}
M.~Petcu and D.~Wirosoetisno.
\newblock Sobolev and Gevrey regularity results for the primitive equations in three space dimensions.
\newblock {\em Applicable Analysis}, 84(8):769--788, 2005.

\bibitem{prodanova2008application}
M.~Prodanova, J.~L. Perez, D.~Syrakov, R.~San~Jose, K.~Ganev, N.~Miloshev, and S.~Roglev.
\newblock Application of mathematical models to simulate an extreme air pollution episode in the Bulgarian city of Stara Zagora.
\newblock {\em Applied Mathematical Modelling}, 32(8):1607--1619, 2008.

\bibitem{simon1986compact}
J.~Simon.
\newblock Compact sets in the space $L^p(0, t; B)$.
\newblock {\em Annali di Matematica pura ed applicata}, 146(1):65--96, 1986.

\bibitem{temam2001navier}
R.~Temam.
\newblock Navier-Stokes Equations: Theory and Numerical Analysis, volume 343.
\newblock {\em American Mathematical Soc.}, 2001.

\bibitem{temam2005some}
R.~Temam and M.~Ziane.
\newblock Some mathematical problems in geophysical fluid dynamics.
\newblock {\em Handbook of mathematical fluid dynamics}, 3:535--658, 2005.

\bibitem{wang2014ocean}
Q.~Wang, W.~Zhou, D.~Wang, and D.~Dong.
\newblock Ocean model open boundary conditions with volume, heat and salinity conservation constraints.
\newblock {\em Advances in Atmospheric Sciences}, 31(1):188--196, 2014.

\bibitem{yeh1975three}
G.-T. Yeh and C.-H. Huang.
\newblock Three-dimensional air pollutant modelling in the lower atmosphere.
\newblock {\em Boundary-Layer Meteorology}, 9(4):381--390, 1975.

\bibitem{zhuk2016source}
S.~Zhuk, T.~T. Tchrakian, S.~Moore, R.~Ord{\'o}{\~n}ez-Hurtado, and R.~Shorten.
\newblock On source-term parameter estimation for linear advection-diffusion equations with uncertain coefficients.
\newblock {\em SIAM Journal on Scientific Computing}, 38(4):A2334--A2356, 2016.

\end{thebibliography}

\end{document}